\documentclass[DIV=14,letterpaper]{scrartcl}

\usepackage{tikz}
\usetikzlibrary{arrows}
\usetikzlibrary{calc}
\usetikzlibrary{matrix}
\usetikzlibrary{patterns}
\usepackage{amsmath,amssymb,amsfonts,amsthm,subfig} 
\usepackage{verbatim}

\usepackage{graphicx}

\def \p{\partial}
\def \D{\mathbb{D}}
\def \DK{\mathbb{D}_K}
\def \M{\mathbb{M}}
\def \MK{\mathbb{M}_K}
\def \Th{\mathcal{T}_h}

\newcommand{\V}[1]{\mbox{\boldmath $ #1 $}}
\newcommand{\bey}{\begin{eqnarray}}
\newcommand{\eey}{\end{eqnarray}}

\newcommand{\beq}{\begin{equation}}
\newcommand{\eeq}{\end{equation}}
\theoremstyle{plain}
\newtheorem{thm}{\hspace{6mm}\textbf{ Theorem}}[section]

\theoremstyle{definition}

\theoremstyle{remark}
\newtheorem{exam}{\hspace{6mm}\textbf{ Example}}[section]

\vspace{5mm}

\begin{document}

\date{}
\title{Anisotropic mesh adaptation for 3D anisotropic diffusion problems with application to fractured
reservoir simulation%
\thanks{The work was supported in part by the National Science Foundation (USA) under Grant 1115118.
}
}
\author{Xianping Li%
\thanks{Department of Mathematics and Statistics, the University of Missouri-Kansas City, Kansas City, MO 64110,
U.S.A. (\textit{lixianp@umkc.edu})}
\and Weizhang Huang%
\thanks{Department of Mathematics, the University of Kansas, Lawrence, KS 66045,
U.S.A. (\textit{whuang@ku.edu})}
}
\maketitle

\vspace{10pt}

\begin{abstract}
Anisotropic mesh adaptation is studied for linear finite element solution of 3D anisotropic diffusion problems.
The $\M$-uniform mesh approach is used, where an anisotropic adaptive mesh is generated as a uniform one
in the metric specified by a tensor. In addition to mesh adaptation,
preservation of the maximum principle is also studied. 
Some new sufficient conditions for maximum principle preservation are developed,
and a mesh quality measure is defined to server as a good indicator.
Four different metric tensors are investigated:
one is the identity matrix,  one focuses on minimizing an error bound, another one on preservation
of the maximum principle, while the fourth combines both. Numerical examples show that these
metric tensors serve their purposes. Particularly, the fourth leads to meshes that improve
the satisfaction of the maximum principle by the finite element solution while concentrating elements
in regions where the error is large. 
Application of the anisotropic mesh adaptation to
fractured reservoir simulation in petroleum engineering is also investigated, 
 where unphysical solutions can occur and mesh adaptation can help improving the
satisfaction of the maximum principle.
\end{abstract}



\vspace{10pt}

\section{Introduction}
\label{Sec-intro}

We are concerned with the linear finite element solution of the three dimensional boundary value problem (BVP) of the diffusion equation
\beq
\label{bvp-pde}
 - \nabla \cdot (\D \, \nabla u )  =  f,  \quad \mbox{ in } \quad \Omega
\eeq
subject to the Dirichlet boundary condition
\beq
\label{bvp-bc}
u = g,  \quad \mbox{ on } \quad \partial \Omega 
\eeq
where $\Omega$ is a bounded polyhedral domain,
$f$ and $g$ are given functions, and $\D$ is the diffusion matrix.
We assume that $\D = \D(\V{x}) $ is a symmetric and uniformly positive
definite matrix-valued function on $\Omega$. It includes both isotropic and anisotropic diffusion
as special examples. In the former case, $\D$ takes the form $\D = \alpha(\V{x}) I$,
where $I$ is the $3 \times 3$ identity matrix and $\alpha = \alpha(\V{x})$ is a scalar function.
In the latter case, on the other hand, $\D$ has not-all-equal eigenvalues at least
on a portion of $\Omega$.  

Anisotropic diffusion problems arise from various branches of science and engineering, including plasma physics
\cite{GL09, GLT07,  GYKL05, NW00, SH07}, petroleum engineering \cite{ABBM98a,ABBM98b,CSW95,MD06},
and image processing \cite{CS00, CSV03, KM09, PM90, Wei98}. When a conventional numerical method
is used to solve the problems, spurious oscillations may occur in computed solutions.
Numerous research has been done for two dimensional (2D) problems; among other works, we mention a few here,
\cite{BKK08,CR73,DDS04,GYKL05,HKL10-2,KKK07,KSS09,LH10,LH13,LSS07,LSSV07,LS08,LePot09,WaZh11,XZ99,YuSh2008,ZZS2013}.
A common approach is to design a proper discretization method and/or a proper mesh so that the numerical solution satisfies the maximum principle (MP). Recently, an anisotropic non-obtuse angle condition was developed in
\cite{Hua11,LH10,LH13,LHQ14}
for the linear finite element solution of both time independent and time dependent anisotropic
diffusion problems to satisfy MP.

On the other hand, much less work has been done for 3D anisotropic diffusion problems. 
Although MP preservation has been studied in general dimensions e.g. in
\cite{BKK08, CR73, GYKL05, KKK07,KSS09, KL95, LH10, LH13}, most of them either consider 
isotropic diffusion or present numerical examples only in 1D and 2D.
For example, only isotropic diffusion is considered in \cite{CR73,KL95}.
It is shown in \cite{Let92} that the 3D Delaunay triangulation does not generally produce
a mesh with which the numerical solution satisfies MP. Mesh conditions are studied in \cite{BKK08}
for a reaction-isotropic-diffusion problem for general dimensions and numerical examples in 1D and 2D are presented.
The difficulty of MP satisfaction for 3D problems is remarked in both \cite{BKK08} and \cite{Let92}.

The objective of this paper is to study the linear finite element solution of 3D anisotropic diffusion problems.
The focus will be on MP preservation and mesh adaptation. Four different metric tensors used in anisotropic mesh
generation will be considered. This study is a 3D extension of the work \cite{LH10}. 
Moreover, new sufficient conditions will be developed for the linear finite element approximation to satisfies MP, and a mesh quality measure
will be defined to provide a useful indication for MP satisfaction.
The mesh quality measure is developed along the approach of \cite{Hua05c}.
But the interested reader is also referred to \cite{Knu03, Knu12} and references therein for
different mesh optimization methods and quality metrics.
Furthermore, the application to fractured reservoir simulation in petroleum engineering will also be investigated,
where unphysical solutions can occur and mesh adaptation can help improving MP satisfaction.

An outline of the paper is given as follows. The linear finite element formulation for BVP (\ref{bvp-pde}) and
(\ref{bvp-bc}) is given in Section~\ref{sec-FEM}. MP preservation and some sufficient conditions will be discussed.
Section~\ref{sec-ANAC} contains the discussion of anisotropic mesh adaptation and four metric tensors.
Numerical examples are given in Section~\ref{sec-exam}, followed by the investigation of application
of anisotropic mesh adaptation to fractured reservoir simulation in Section~\ref{sec-app}.
Conclusions are drawn in Section~\ref{sec-con}.

\section{Finite Element Formulation}
\label{sec-FEM}

In this section we briefly describe the piecewise linear finite element discretization for
BVP (\ref{bvp-pde}) and (\ref{bvp-bc}) and state a few properties of the discretization.

Let $U_g = \{ v \in H^1(\Omega) \; | \: v|_{\p \Omega} = g\}$. The weak formulation of
BVP (\ref{bvp-pde}) and (\ref{bvp-bc}) is to find $u \in U_g$ such that
\[
\int_\Omega (\nabla v)^T \D \nabla u \, d \V{x} = \int_\Omega f v \, d \V{x}, \quad \forall v \in U_0 .
\]
For the finite element discretization, we assume that a tetrahedral mesh $\Th$ is given
on $\Omega$. Let $g^h$ be the piecewise linear interpolation of $g$ on the boundary $\partial \Omega$
and $U_{g^h}^h$ be the piecewise linear
finite element space on $\Th$ with the boundary data $g^h$. Then, the finite element formulation
for BVP (\ref{bvp-pde}) and (\ref{bvp-bc})  is to find $u^h \in U_{g^h}^h$ such that
\beq
\label{fem-form}
\int_{\Omega} (\nabla v^h)^T \; \D \nabla u^h d\V{x} =
 \int_{\Omega} f \, v^h d\V{x}, \quad \forall v^h \in U_0^h .
\eeq
This discretization is standard and it is expected that $u^h$ converges to $u$
at a rate of second order in the $L^2$ norm and first order in $H^1$ norm. We take the reference element $\hat K$ to be a unitary equilateral  tetrahedron and denote an element in the mesh $\Th$ by $K$. Moreover, 
denote by $\hat{\V{q}}_i$ ($i=1, ..., 4$) the unit inward normal to the face facing the $i^{\text{th}}$ vertex
of $\hat K$. Let $F_K$ be the affine mapping from $\hat K$ to $K$ and $F_K'$ the Jacobian matrix of $F_K$.

\begin{thm}[Li and Huang  \cite{LH10}]
\label{thm-mesh-map}
If the mesh satisfies
\beq
\label{meshcnd-2}
\hat{\V{q}}_i^T (F_K')^{-1} \DK (F_K')^{-T} \hat{\V{q}}_j \le 0,
\quad \forall i \ne j,\; i,j=1,...,4,\; \forall K \in \Th
\eeq
where $\DK$ is the average of $\D$ over $K$,
then the linear finite element scheme (\ref{fem-form}) for solving BVP (\ref{bvp-pde}) and (\ref{bvp-bc})
satisfies the maximum principle,
\[
f \le 0 \text{ in }\Omega \quad \Longrightarrow \quad
\max_{\V{x} \in \Omega \cup \partial \Omega} u^h(\V{x}) = \max_{\V{x} \in \partial \Omega} u^h(\V{x}) .
\]
\end{thm}

Preserving MP is crucial to avoiding artificial oscillations in the computed solution.
It is thus interesting to know what meshes satisfy the \textit{ anisotropic non-obtuse angle condition} (\ref{meshcnd-2}).
Obviously, a sufficient condition is
\beq
\label{meshcnd-3}
(F_K')^{-1} \DK (F_K')^{-T} = c_K I,\quad \forall K \in \Th
\eeq
where $c_K$ is a positive scalar constant on $K$ and $I$ is the $3 \times 3$ identity matrix.

A weaker condition is stated in the following theorem, where a general case for the $d$ space dimension
($d \ge 2$) is considered.

\begin{thm}
\label{lem:quality}
A sufficient condition for \eqref{meshcnd-2} is
\begin{equation}
\label{meshcnd-4}
\left \| \frac{(F_K')^{-1} \DK (F_K')^{-T}}{\det ( (F_K')^{-1} \DK (F_K')^{-T} )^{\frac{1}{d}} }
- I \right \| \le \frac{1}{d}
\end{equation}
or 
 \beq 
 \label{meshcnd-5}
 \frac{\|(F'_K)^{-1} \DK (F'_K)^{-T}\|}{\det \Big((F'_K)^{-1} \DK (F'_K)^{-T}\Big)^{\frac{1}{d}}} \leq \min
 \left (1+\frac{1}{d},\;  \Big(1-\frac{1}{d}\Big)^{-\frac{1}{d-1}}\right ) .
 \eeq
\end{thm}

\begin{proof}
For $i\neq j$, the left-hand side of (\ref{meshcnd-2}) can be rewritten as
\begin{align*}
&\hat{\V{q}}_i^T (F_K')^{-1} \DK (F_K')^{-T} \hat{\V{q}}_j
\\
= & \hat{\V{q}}_i^T \left [ \det ( (F_K')^{-1} \DK (F_K')^{-T} )^{\frac{1}{d}} I
 + (F_K')^{-1} \DK (F_K')^{-T}
-  \det ( (F_K')^{-1} \DK (F_K')^{-T} )^{\frac{1}{d}} I \right ] \hat{\V{q}}_j
\\
= & - \frac{1}{d}  \det ( (F_K')^{-1} \DK (F_K')^{-T} )^{\frac{1}{d}} 
+ \hat{\V{q}}_i^T \left [ (F_K')^{-1} \DK (F_K')^{-T}
- \det ( (F_K')^{-1} \DK (F_K')^{-T} )^{\frac{1}{d}}  I
\right ] \hat{\V{q}}_j ,
\end{align*}
where we have used $\hat{\V{q}}_i^T \hat{\V{q}}_j = - \frac{1}{d}$ for the equilateral simplex $\hat K$.
From this, we have
\begin{align*}
& \hat{\V{q}}_i^T (F_K')^{-1} \DK (F_K')^{-T} \hat{\V{q}}_j
\\
& \qquad \le - \frac{1}{d}  \det ( (F_K')^{-1} \DK (F_K')^{-T} )^{\frac{1}{d}} 
+ \| (F_K')^{-1} \DK (F_K')^{-T} - \det ( (F_K')^{-1} \DK (F_K')^{-T} )^{\frac{1}{d}}  I \| ,
\end{align*}
where $\| \cdot \|$ is the matrix 2-norm. Thus, a sufficient condition for (\ref{meshcnd-2}) is
\[
- \frac{1}{d}  \det ( (F_K')^{-1} \DK (F_K')^{-T} )^{\frac{1}{d}}
+ \| (F_K')^{-1} \DK (F_K')^{-T} - \det ( (F_K')^{-1} \DK (F_K')^{-T} )^{\frac{1}{d}}  I \| \le 0 .
\]
From this we can obtain (\ref{meshcnd-4}).

 Denote the eigenvalues of $(F'_K)^{-1} \DK (F'_K)^{-T}$ by $\lambda_1 \geq \cdots \geq \lambda_d > 0$. Then condition \eqref{meshcnd-5} is equivalent to
 \beq
 \label{meshcnd-6}
 \frac{\lambda_1}{(\lambda_1 \cdots \lambda_d)^{\frac{1}{d}}} \leq \min \Big\{1+\frac{1}{d}, \Big(1-\frac{1}{d}\Big)^{-\frac{1}{d-1}} \Big\}.
 \eeq
 Meanwhile, \eqref{meshcnd-4} is equivalent to
 \bey
 \nonumber
 && \left\Vert \frac{1}{(\lambda_1 \cdots \lambda_d)^{\frac{1}{d}}}
 \begin{bmatrix}
  \lambda_1 & & \\
  & \ddots & \\
  & & \lambda_d
 \end{bmatrix} 
 -
 \begin{bmatrix}
  1 & & \\
  & \ddots & \\
  & & 1
 \end{bmatrix}	 	 
 \right\Vert \leq \frac{1}{d} \\
 \nonumber
\Longleftrightarrow &&  \left| \frac{\lambda_i}{(\lambda_1 \cdots \lambda_d)^{\frac{1}{d}}}-1 \right| 
 \leq \frac{1}{d} \qquad i=1, \cdots, d \\
 \label{meshcnd-7}
\Longleftrightarrow &&  1-\frac{1}{d} \leq \frac{\lambda_d}{(\lambda_1 \cdots \lambda_d)^{\frac{1}{d}}} 
 \leq \frac{\lambda_1}{(\lambda_1 \cdots \lambda_d)^{\frac{1}{d}}} \leq 1+\frac{1}{d}.
 \eey
 Thus, \eqref{meshcnd-6}, or   \eqref{meshcnd-5}, implies the right inequality of \eqref{meshcnd-7}. 
 
 To prove the left inequality of \eqref{meshcnd-7}, we start from \eqref{meshcnd-6} and have
 \beq \nonumber
 \Big(1-\frac{1}{d}\Big)^{-\frac{1}{d-1}} \geq \frac{\lambda_1}{(\lambda_1 \cdots \lambda_d)^{\frac{1}{d}}} \geq \frac{\lambda_1}{\lambda_d^{\frac{1}{d}}\lambda_1^{\frac{d-1}{d}}} = \Big(\frac{\lambda_1}{\lambda_d}\Big)^{\frac{1}{d}},
 \eeq
 or 
 \beq
 \nonumber
 \Big(\frac{\lambda_d}{\lambda_1}\Big)^{\frac{1}{d}} \geq \Big(1-\frac{1}{d}\Big)^{\frac{1}{d-1}}.
 \eeq
 Using this, we have
 \beq
 \nonumber
 	\frac{\lambda_d}{(\lambda_1 \cdots \lambda_d)^{\frac{1}{d}}} \geq \frac{\lambda_d}{\lambda_d^{\frac{1}{d}}\lambda_1^{\frac{d-1}{d}}} = \Big(\frac{\lambda_d}{\lambda_1}\Big)^{\frac{d-1}{d}} \geq 1-\frac{1}{d}
 \eeq
 Thus, \eqref{meshcnd-5} implies the left inequality of \eqref{meshcnd-7}. 
\end{proof}

Notice that the bound in \eqref{meshcnd-5} has the value
\[
\min\Big\{1+\frac{1}{d}, \Big(1-\frac{1}{d}\Big)^{-\frac{1}{d-1}}\Big\} =
\begin{cases}
1.5,&\quad \text{ for 2D ($d=2$)}\\
\Big(\frac{3}{2}\Big)^{\frac{1}{2}} \approx 1.225 ,& \quad \text{ for 3D ($d=3$)}.
\end{cases}
\]

Moreover, (\ref{meshcnd-3}) is sufficient for (\ref{meshcnd-5}) because 
(\ref{meshcnd-3}) implies that all of the eigenvalues of $(F'_K)^{-1} \DK (F'_K)^{-T}$ are equal to $c_K$
and thus the left-hand side of (\ref{meshcnd-5}) is equal to one which is less than the right-hand side.
This means that (\ref{meshcnd-5}) is weaker than (\ref{meshcnd-3}).

Unfortunately, (\ref{meshcnd-5}) is still stronger than (\ref{meshcnd-2}).
Consider the triangular and tetrahedral elements in Fig.~\ref{fig:mesh-example}
and the case with $\D = I$. It is easy to see that the elements are non-obtuse and satisfy
(\ref{meshcnd-2}). A direct calculation shows that 
\[
\max_K \frac{\| (F_K')^{-1} \DK (F_K')^{-T}\|}{\det ( (F_K')^{-1} \DK (F_K')^{-T} )^{\frac{1}{d}} }=
\begin{cases}
1.732, & \quad \text{ for 2D elements in Fig.~\ref{fig:mesh-example}(a)}\\
2.151, & \quad \text{ for 3D elements in Fig.~\ref{fig:mesh-example}(b)}
\end{cases}
\]
which violates (\ref{meshcnd-5}).
Nevertheless, as will be seen in the next section, the conditions (\ref{meshcnd-3}) and (\ref{meshcnd-5})
provide useful guidelines for generating meshes that improve the MP satisfaction of the scheme.

\begin{figure}[hbt]
\centering
\hspace{1mm}
\hbox{
\begin{minipage}[b]{2.5in}
\centering
\includegraphics[scale=0.325]{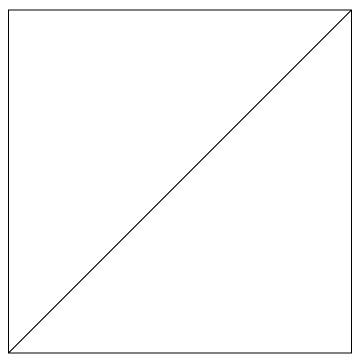}
\centerline{(a) 2D}
\end{minipage}
\hspace{1mm}
\begin{minipage}[b]{2.5in}
\centering
\includegraphics[scale=0.2]{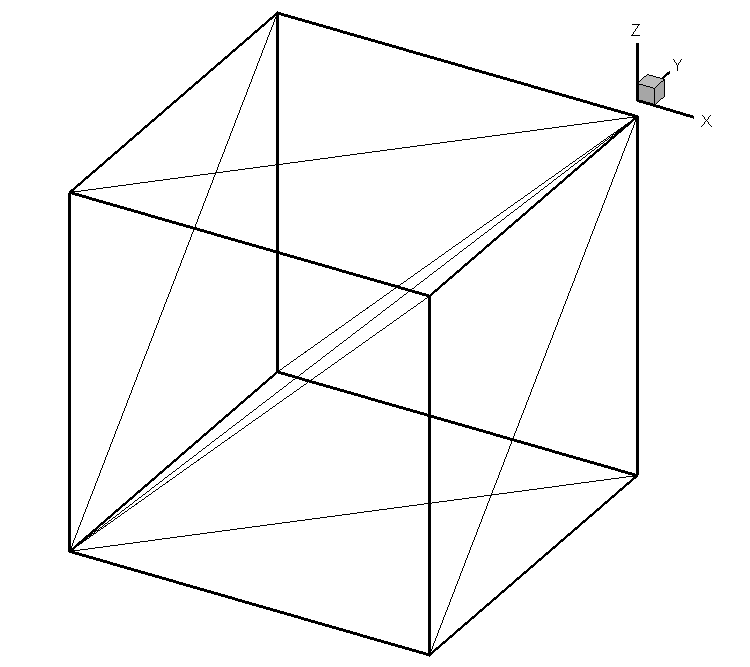}
\centerline{(b) 3D}
\end{minipage}
}
\caption{Examples of triangular and tetrahedral elements.}
\label{fig:mesh-example}
\end{figure}

\section{Anisotropic mesh adaptation}
\label{sec-ANAC}

In this section we study anisotropic mesh adaptation for the anisotropic diffusion problem
(\ref{bvp-pde}) and (\ref{bvp-bc}). We use the so-called $\M$-uniform mesh approach
\cite{Hua05b,Hua06} where an adaptive mesh is viewed as a uniform one in the metric specified
by a tensor $\M = \M(\V{x})$ which is assumed to be symmetric and uniformly positive
definite on $\Omega$. For the moment, we assume that $\M$ has been given. Its choice will be
discussed later. It is shown in \cite{Hua06} that an $\M$-uniform tetrahedral mesh $\Th$ satisfies
\begin{align}
\label{eq-1}
&|K| \det(\MK)^{\frac{1}{2}}   =  \frac{\sigma_h}{N}, \quad \forall K \in \Th \\
\label{ali-1}
&\frac{1}{3} \mbox{tr} \left ( (F_K')^{-1} (\MK)^{-1} (F_K')^{-T} \right )  = 
\mbox{det} \left ( (F_K')^{-1} (\MK)^{-1} (F_K')^{-T} \right )^{\frac{1}{3}}, \quad \forall K \in \Th
\end{align}
where 
\beq
\MK = \frac{1}{|K|} \int_K \M(\V{x}) d \V{x},\quad
\sigma_h = \sum_{K \in \Th} \det(\MK)^{\frac{1}{2}} |K|.
\eeq
Condition (\ref{eq-1}) is called as \textit{the equidistribution condition} which requires all of the elements to
have the same size in the metric $\MK$ and therefore determines the size of $K$ from $\det(\MK)^{\frac{1}{2}}$.
On the other hand, (\ref{ali-1}) is called \textit{the alignment condition} which requires $K$, measured
in the metric $\MK$, to be similar to the reference element $\hat K$ that is measured in the Euclidean metric
and thus controls the shape and orientation of $K$. It can also be shown that
the principal axes of the circumscribed ellipsoid of $K$ are required to be parallel to
the eigenvectors of $\MK$ while their lengths are reciprocally proportional to
the square roots of the respective eigenvalues \cite{Hua06}.

$\M$-uniform or nearly $\M$-uniform meshes can be generated using various strategies;
see \cite[Section 4]{Hua06} for more detailed discussion. We mention just a few here, 
the Delaunay triangulation method\cite{BGHLS97,BGM97,CHMP97,PVMZ87},
the advancing front method \cite{GS98,MA14},
the bubble mesh method \cite{YS00},
the combination of combining refinement, local modification, and
local node movement \cite{ABHFDV02,BH96,DVBFH02,HDBAFV00,Los14},
and the variational method \cite{Hua01b}.
Particularly, we mention two C++ codes which can directly take the user supplied metric tensor.
One is BAMG (Bidimensional Anisotropic Mesh Generator) for 2D meshes developed by Hecht \cite{bamg} based on
the Delaunay triangulation and local node movement.
The other is MMG3D (Anisotropic Tetrahedral Remesher/Moving Mesh Generation) developed
by Dobrzynski and Frey \cite{mmg3d} based on refinement and local node movement.
The latter is used in our computation.

A key component of the $\M$-uniform mesh approach is to define the metric tensor. We consider
four choices here. The first one is
\begin{equation}
\label{M-uniform}
\M_{id} = I, \quad \forall \V{x} \in \Omega.
\end{equation}
This is the simplest choice, and the resulting meshes are uniform or nearly uniform.

The second choice is based on linear interpolation error. It is known that the error for
the linear finite element solution to (\ref{bvp-pde}) and (\ref{bvp-bc}) is bounded by
the error in the piecewise linear interpolation of the exact solution. Thus, it is reasonable
to define the metric tensor based on linear interpolation error. A metric tensor based on
minimization of the $H^1$ semi-norm of linear interpolation error is given \cite{Hua05b} by
\beq
\label{M-adap}
\M_{adap}(K) = \Big \| I + \frac{1}{\alpha_h} | H_K(u^h) | \Big \| ^{\frac{2}{5}}
\det\left( I+\frac{1}{\alpha _{h}}|H_K(u^h)|\right) ^{-\frac{1}{5}}
\left[ I+\frac{1}{\alpha _{h}}|H_K(u^h)|\right] ,
\quad \forall K \in \Th
\eeq
where $u^h$ is the finite element solution, $H_K(u^h)$ is a recovered Hessian of $u^h$ over $K$,
$|H_K(u^h)|$ is the eigen-decomposition of $H_K(u^h)$ with the eigenvalues being replaced
by their absolute values, and $\alpha_h$ is defined implicitly through
\beq
\label{alphah}
\sum_{K\in \Th} |K| \sqrt{\det(\M_{adap}(K))} = 2 | \Omega | .
\eeq 
This equation determines $\alpha_h$ uniquely and can be solved, for instance, by the bisection method.
Moreover, the choice of $\alpha_h$ in this way concentrates roughly 50\% of the mesh points
in the region where the solution changes significantly.

The third and fourth choices are related to the diffusion matrix.
Recall that (\ref{ali-1}) requires that all of the eigenvalues of the matrix $(F_K')^{-1} (\MK)^{-1} (F_K')^{-T}$
be equal to each other. It is mathematically equivalent to
\beq
\label{ali-2}
(F_K')^{-1} (\MK)^{-1} (F_K')^{-T} = \tilde{c}_K I
\eeq
for some constant $\tilde{c}_K$ on $K$. Comparing this with (\ref{meshcnd-3}) suggests that we choose
the metric tensor as
\beq
\M(K) = \theta_K \DK^{-1}, \quad \forall K \in \Th
\label{M-DMP-0}
\eeq
where $\theta_K$ is an arbitrary positive piecewise constant function and $\DK$ is the average
of $\D$ over $K$. From (\ref{ali-2}), it is not difficult to see that any $\M$-uniform mesh associated
with this metric tensor satisfies (\ref{meshcnd-3}) and therefore, the finite element solution $u^h$
satisfies MP.

Then, the third choice is
\beq
\M_{DMP}(K) = \DK^{-1}, \quad \forall K \in \Th
\label{M-DMP}
\eeq
which corresponds to $\theta_K = 1$.
For the fourth choice, we take advantage of the arbitrariness of $\theta_K$ in \eqref{M-DMP-0} and choose it to minimize
a bound of $H^1$ semi-norm of linear interpolation error. This way we can combine the desire of
preserving MP with mesh adaptation.
The metric tensor reads as (cf. \cite{LH10})
\beq
\label{M-DMP+adap}
\M_{DMP+adap}(K) = \left ( 1+ \frac{1}{\alpha_h}  B_K \right )^{\frac{2}{5}}
\mbox{det} \left ( \DK \right )^{\frac{1}{3}} \DK^{-1},
\eeq
where
\begin{align*}
& B_K = \mbox{det} \left ( \DK \right )^{-\frac{1}{3}}
 \|\DK^{-1} \|\cdot   \|\DK |H_K(u^h)| \|^2,
\\
& \alpha_h = \left ( \frac{1}{|\Omega|} \sum_{K \in \Th} |K| B_K^{\frac{3}{5}} \right )^{\frac{5}{3}}.
\end{align*}
The $\M$-uniform meshes associated with $\M_{DMP}$ and $\M_{DMP+adap}$ satisfy the alignment condition \eqref{ali-1}, and the mesh elements are aligned along the principal diffusion direction defined as the direction of the eigenvector corresponding to the largest eigenvalue of the diffusion matrix $\D$.

It should be pointed out that in practical computation, it is difficult, if not impossible, to generate a perfect $\M$-uniform
mesh that satisfy the conditions (\ref{eq-1}) and (\ref{ali-1}).
Thus, it makes sense to measure how far a given mesh is from satisfying these conditions.
From (\ref{eq-1}) and (\ref{ali-1}), we can define the equidistribution and alignment measures as
\bey
\label{meas-eq}
Q_{eq}(K) & = & \frac{N |K| \det(\MK)^{\frac{1}{2}}}{\sigma_h}, \\
\label{meas-ali}
Q_{ali}(K) & = & \frac{ \| (F_K')^{-1} (\MK)^{-1} (F_K')^{-T} \|}
				{\mbox{det} \left ( (F_K')^{-1} (\MK)^{-1} (F_K')^{-T} \right )^{\frac{1}{3}}}.
\eey
It is not difficult to show that the maximum norm of both functions, $\| Q_{eq} \|_{L^\infty}$
and $\| Q_{ali} \|_{L^\infty}$, has the range $[1, \infty)$.
Moreover, $\| Q_{eq} \|_{L^\infty} = 1$ and $\| Q_{ali} \|_{L^\infty} = 1$ imply a perfect $\M$-uniform mesh.
The larger $\| Q_{eq} \|_{L^\infty}$ and $\| Q_{ali} \|_{L^\infty}$ are, the farther the mesh is away from being
$\M$-uniform. It is worth emphasizing that only $Q_{ali}$ is related to MP satisfaction.
The quantity $Q_{eq}$ indicates how evenly the error is distributed among the elements.

It is interesting to point out that the definition (\ref{meas-ali}) is slightly different from
a more common definition \cite{Hua06} where the trace of  $(F_K')^{-1} (M_K)^{-1} (F_K')^{-T}$ is used,
\beq
\label{meas-ali2}
\widetilde{Q}_{ali}(K) = \frac{\frac{1}{3}\mbox{tr} \left ( (F_K')^{-1} (M_K)^{-1} (F_K')^{-T} \right )}
				{\mbox{det} \left ( (F_K')^{-1} (M_K)^{-1} (F_K')^{-T} \right )^{\frac{1}{3}}}.
\eeq
These two definitions are equivalent since the trace and 2-norm of a positive definite matrix are equivalent.
The main motivation for which we use (\ref{meas-ali}) is that the condition (\ref{meshcnd-5})
can now be rewritten as
\beq
\| Q_{ali}\|_{L^\infty} \le 
\min\Big\{1+\frac{1}{d}, \Big(1-\frac{1}{d}\Big)^{-\frac{1}{d-1}}\Big\} .
\label{meas-ali3}
\eeq
As mentioned before, this condition is hard to satisfy. Indeed, for meshes shown in Fig.~\ref{fig:mesh-example},
we have $\| Q_{ali}\|_{L^\infty} = 1.732$ (2D) and 2.151 (3D). 
Although they violate (\ref{meas-ali3}), the meshes can still be considered to be close to satisfying the alignment condition
since $\| Q_{ali}\|_{L^\infty}$ is relatively small. Moreover, (\ref{meas-ali3}) shows a connection between
the alignment condition and the preservation of MP.

\section{Numerical examples}
\label{sec-exam}

In this section we present two three-dimensional examples to demonstrate the performance of the anisotropic
mesh adaptation strategy described in the previous section with the four metric tensors.
An iterative procedure for solving PDEs using anisotropic adaptive meshes is shown in Fig.~\ref{procedure}.
The PDE is first solved on the current mesh using piecewise linear finite elements.
Then, the metric tensor $\M$ is computed, followed by the generation of a new mesh for the metric tensor
using MMG3D. 
This procedure is repeated a few times untill the quliaty measures \eqref{meas-eq} and \eqref{meas-ali} cannot not be improved further. In our computations, the procedure is repeated five times and numerical experiment shows that there is no significant improvement
in the results when more iterations are used.
The computations are performed in the framework of an open source finite element software, FreeFem++ developed by Hecht \cite{FFem}, with the linear solver being chosen as the conjugate gradient method with tolerance $10^{-15}$.
Moreover, the Hessian of the finite element solution $u^h$ used in the computation of the metric tensors
is computed using the ``mshmet'' library embedded in FreeFem++. 
The ``mshmet'' library first approximates the gradient (first derivatives) using the least squares fitting, then uses the approximated gradient in the least squares fitting for the Hessian matrix (second derivatives).

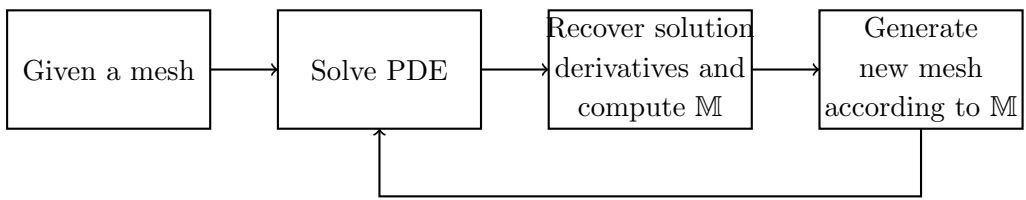
\begin{figure}[tbh]
\centering
\tikzset{my node/.code=\ifpgfmatrix\else\tikzset{matrix of nodes}\fi}
\begin{tikzpicture}[every node/.style={my node},scale=0.45]
\draw[thick] (0,0) rectangle (6,3.5);
\draw[thick] (8,0) rectangle (14,3.5);
\draw[thick] (16,0) rectangle (22,3.5);
\draw[thick] (24,0) rectangle (30,3.5);
\draw[->,thick] (6,1.75)--(8,1.75);
\draw[->,thick] (14,1.75)--(16,1.75);
\draw[->,thick] (22,1.75)--(24,1.75);
\draw[->,thick] (27,0)--(27,-2)--(11,-2)--(11,0);
\node (node1) at (3,1.75) {Given a mesh\\};
\node (node2) at (11,1.75) {Solve PDE\\};
\node (node3) at (19,1.75) {Recover solution\\ derivatives and\\ compute $\M$\\};
\node (node4) at (27,1.75) {Generate\\ new mesh\\ according to $\M$\\};
\end{tikzpicture}
\caption{An iterative procedure for adaptive mesh solution of PDEs.}
\label{procedure}
\end{figure}

The meshes associated with the four metric tensors will be denoted with the same notation
for the metric tensors. For instance, a mesh associated with $\M_{adap}$ will be called an $\M_{adap}$ mesh.
We consider the diffusion matrix in the form
\beq
\D = \begin{bmatrix} \sin \phi \cos \theta & -\sin \theta & \cos \phi \cos \theta \\ \sin \phi \sin \theta & \cos \theta & \cos \phi \sin \theta \\ \cos \phi & 0 & -\sin \phi \end{bmatrix} 
\begin{bmatrix} k_1 & 0 & 0 \\ 0 & k_2 & 0 \\ 0 & 0 & k_3 \end{bmatrix}
\begin{bmatrix} \sin \phi \cos \theta & -\sin \theta & \cos \phi \cos \theta \\ \sin \phi \sin \theta & \cos \theta & \cos \phi \sin \theta \\ \cos \phi & 0 & -\sin \phi \end{bmatrix}^T ,
\label{df-eigen}
\eeq
where  $k_1$ is the dominant eigenvalue, $\phi$ is the angle between the principal diffusion direction
and the positive $z$-axis, and $\theta$ is the angle between the projection of the principal diffusion vector
on the $xy$-plane and the positive $x$-axis.

\begin{exam}
\label{ex1}

For the first example, we would like to compare the $L^2$-norm of solution errors obtained from different meshes and check the violation/satisfaction of MP.
We consider problem (\ref{bvp-pde}) in the unit cube $\Omega=(0,1)^3$.
$\D$ is chosen in the form of (\ref{df-eigen}) with $\phi=-\pi/4$, $\theta=5\pi/6$, $k_1=100$, $k_2=10$, and $k_3=1$. 
The source function $f$ and the boundary function $g$ are chosen such that the exact solution is given by
\beq
\label{ex1-ue}
u = e^{-100((x-0.5)^2+(y-0.5)^2-0.1^2)}+z^2. 
\eeq
The continuous problem satisfies MP and the solution is in the interval $(0,1+e]$.

The exact solution in the domain $\Omega$ and on the cross-sections $x=0.5$ and $z=0.5$ is shown
in Fig.~\ref{ex1-soln}. Fig.~\ref{ex1-mesh} shows the four types of meshes with a view of the inner mesh by
cutting a corner of the domain. Fig.~\ref{ex1-mesh2} shows the meshes on the cross-section $x=0.5$.
The elements in the $\M_{DMP}$ mesh are aligned well along the primary diffusion direction $\theta=5\pi/6$, while the elements in the $\M_{id}$ mesh are aligned along the direction of $\theta=\pi/4$. The elements in the $\M_{adap}$ mesh concentrate around the central region where the solution changes rapidly. The elements in the $\M_{DMP+adap}$
mesh not only are aligned along the primary diffusion direction but also concentrate around the central region,
which is a result of combining MP preservation and adaptivity.

Table~\ref{ex1-result} shows the $L^2$-norm of the solution error obtained on $\M_{id}$, $\M_{adap}$,
$\M_{DMP}$, and $\M_{DMP+adap}$ meshes. The minimal value in the computed solution,
which is denoted by $u_{min}$ and indicates the undershoots and violation of MP,
is also reported in the table. The convergence history of the $L^2$-norm of the error is shown in Fig.~\ref{ex1-conv}.
It is clear that the error converges at a rate of second order as the mesh is refined.
Moreover, the numerical solution obtained using the $\M_{adap}$ mesh has the smallest error,
which is consistent with the fact that $\M_{adap}$ is formulated based on the minimization
of the interpolation error (and the finite element error). The solution obtained using the $\M_{DMP}$
mesh has the largest error due to the fact that the mesh is designed to satisfy MP,
which is generally in conflict with error reduction.  The error of the solution obtained using $\M_{DMP+adap}$
is between those using $\M_{adap}$ and $\M_{DMP}$, which is expected from the construction of
$\M_{DMP+adap}$ and shows a good balance between mesh adaptivity and MP preservation.

It can also be observed from Table~\ref{ex1-result} that the numerical solution obtained from
$\M_{id}$ and $\M_{adap}$ violate MP even for very fine meshes. On the other hand,
the numerical solutions obtained from $\M_{DMP}$ and $\M_{DMP+adap}$ violates MP
for coarse meshes but satisfy MP when the mesh is sufficiently fine.
This indicates that $\M_{DMP}$ and $\M_{DMP+adap}$ can help improving the satisfaction of
MP for the finite element solution.

\begin{figure}[hbt]
\centering
\includegraphics[scale=0.25]{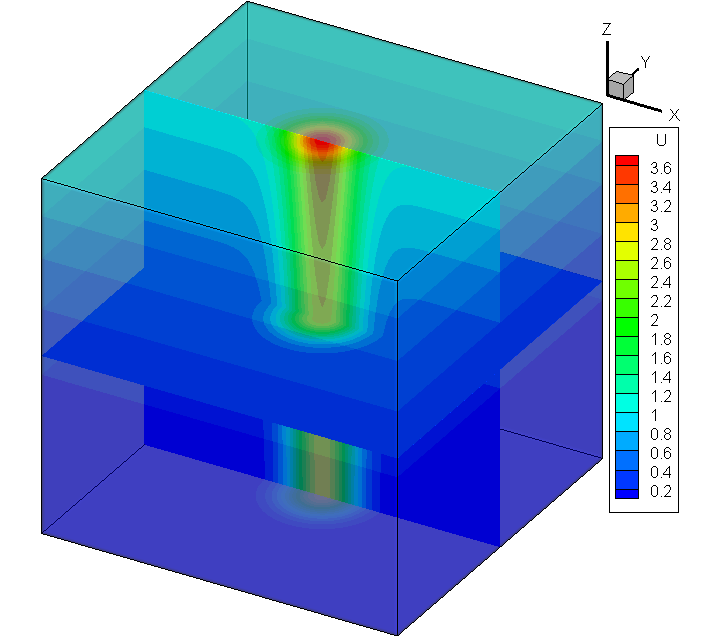}
\caption{Example \ref{ex1}. Exact solution in the domain $\Omega$ and on the cross-sections $y=0.5$ and $z=0.5$.}
\label{ex1-soln}
\end{figure}

\begin{figure}[hbt]
\centering
\hbox{
\hspace{1mm}
\begin{minipage}[b]{2.5in}
\includegraphics[scale=0.25]{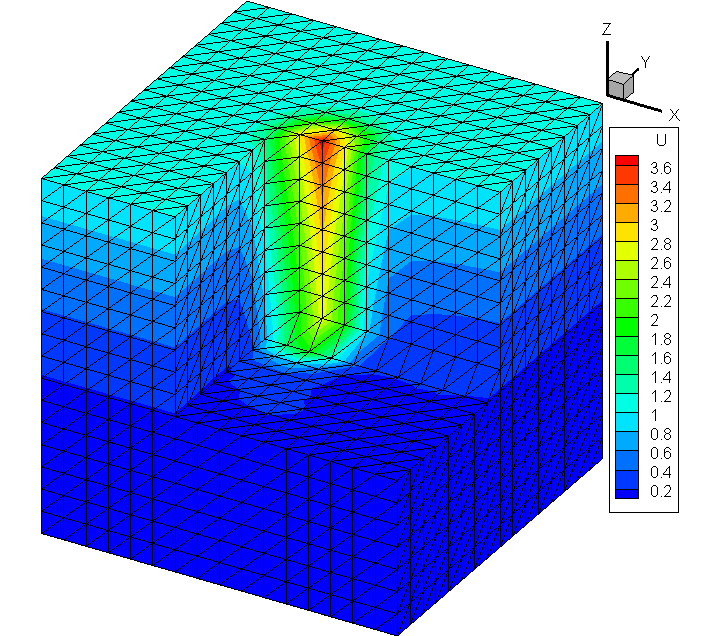}
\centerline{(a): $\M_{id}$ mesh, $N=24,576$}
\end{minipage}
\hspace{1mm}
\begin{minipage}[b]{2.5in}
\includegraphics[scale=0.25]{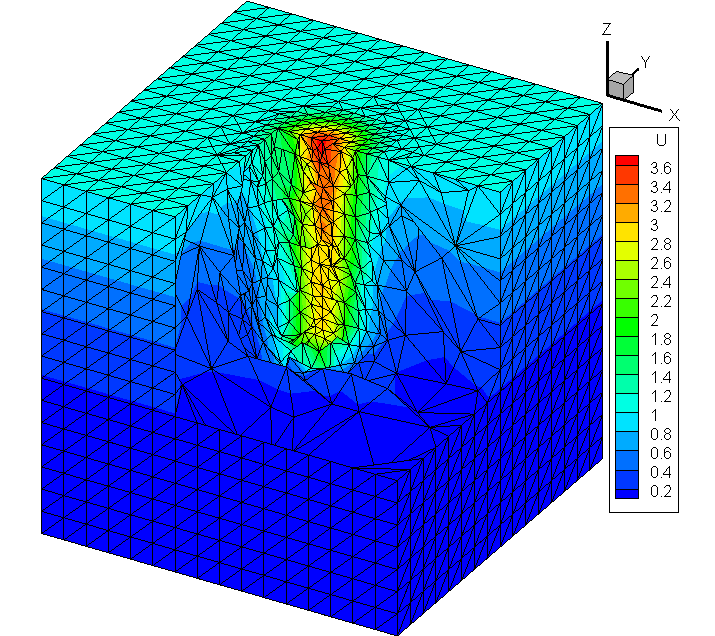}
\centerline{(b): $\M_{adap}$ mesh, $N=17,304$}
\end{minipage}
}
\vspace{5mm}
\hbox{
\hspace{1mm}
\begin{minipage}[b]{2.5in}
\includegraphics[scale=0.25]{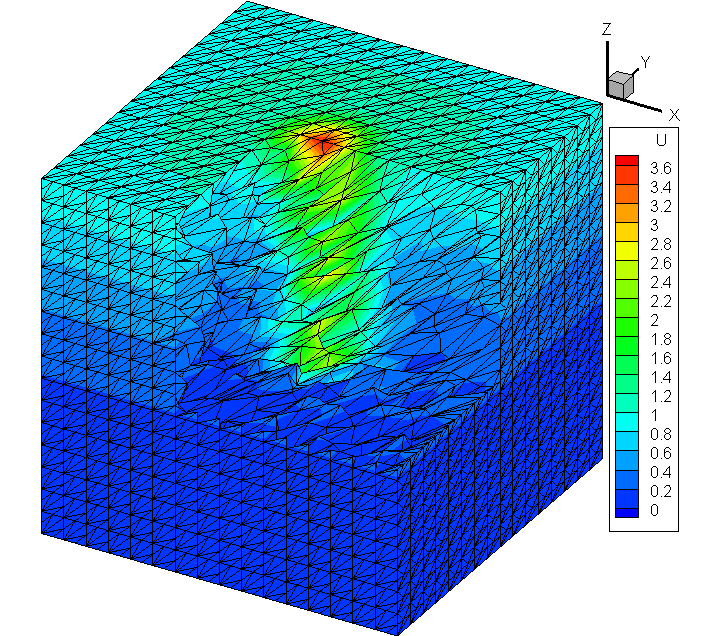}
\centerline{(c): $\M_{DMP}$ mesh, $N=29,023$}
\end{minipage}
\hspace{1mm}
\begin{minipage}[b]{2.5in}
\includegraphics[scale=0.25]{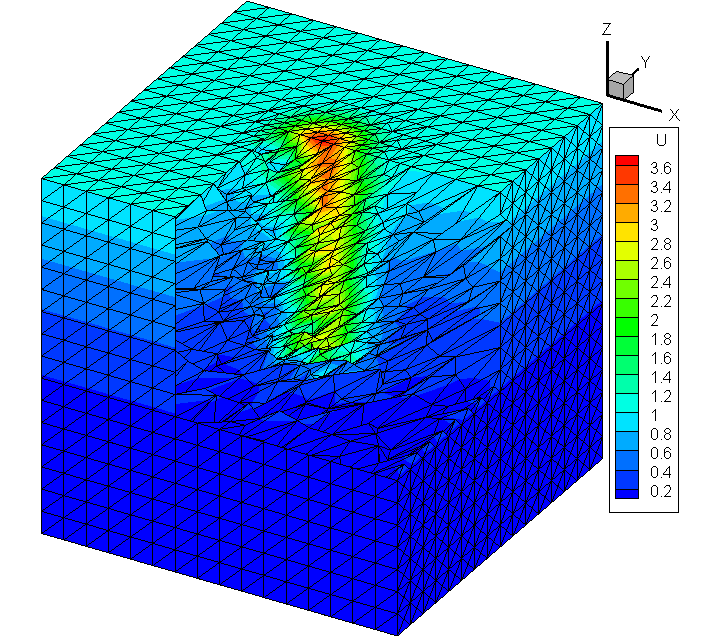}
\centerline{(d): $\M_{DMP+adap}$ mesh, $N=20,099$}
\end{minipage}
}
\caption{Example \ref{ex1}. Different meshes with a view of the inner mesh by cutting a corner.}
\label{ex1-mesh}
\end{figure}

\begin{figure}[hbt]
\centering
\hbox{
\hspace{1mm}
\begin{minipage}[b]{2.5in}
\includegraphics[scale=0.25]{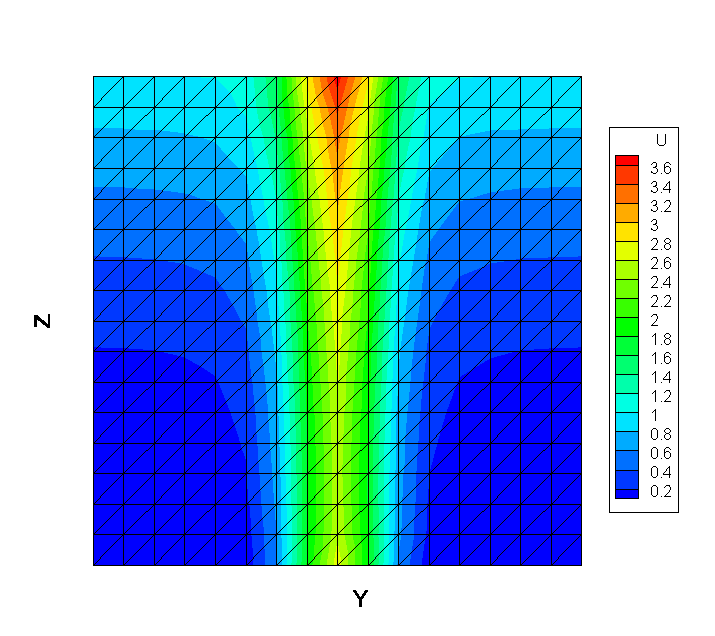}
\centerline{(a): $\M_{id}$ mesh}
\end{minipage}
\hspace{1mm}
\begin{minipage}[b]{2.5in}
\includegraphics[scale=0.25]{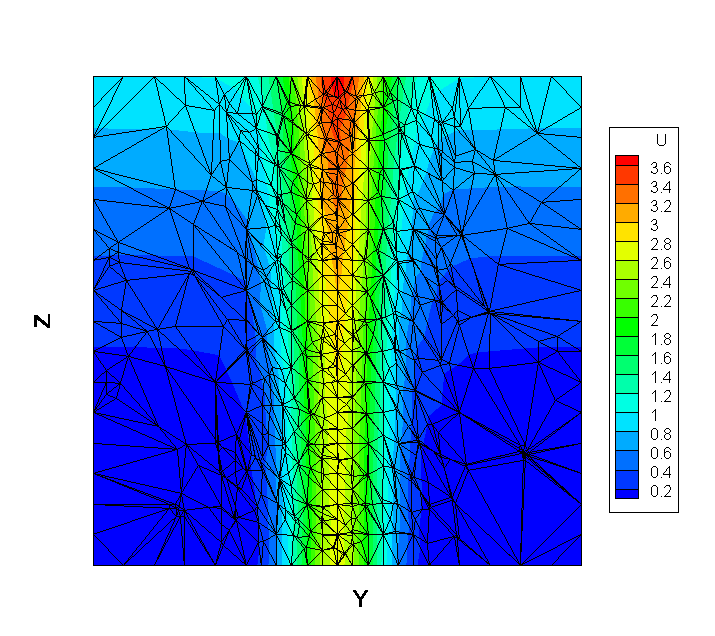}
\centerline{(b): $\M_{adap}$ mesh}
\end{minipage}
}
\vspace{5mm}
\hbox{
\hspace{1mm}
\begin{minipage}[b]{2.5in}
\includegraphics[scale=0.25]{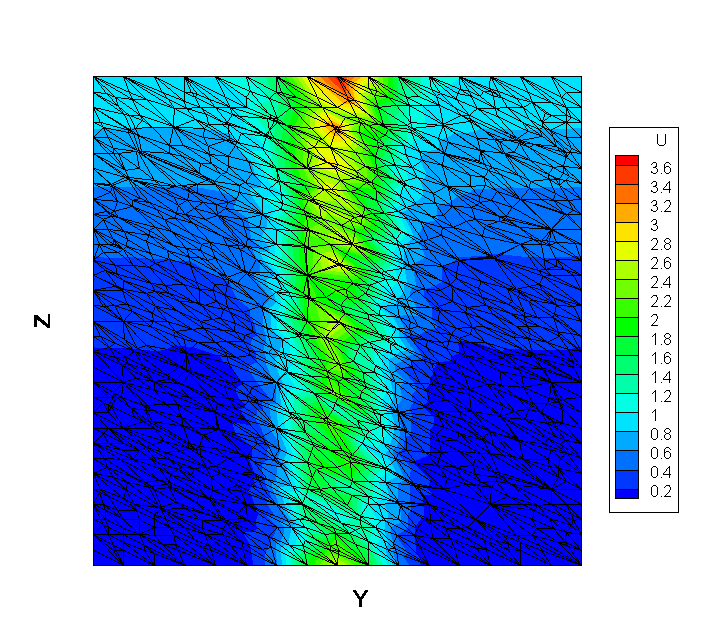}
\centerline{(c): $\M_{DMP}$ mesh}
\end{minipage}
\hspace{1mm}
\begin{minipage}[b]{2.5in}
\includegraphics[scale=0.25]{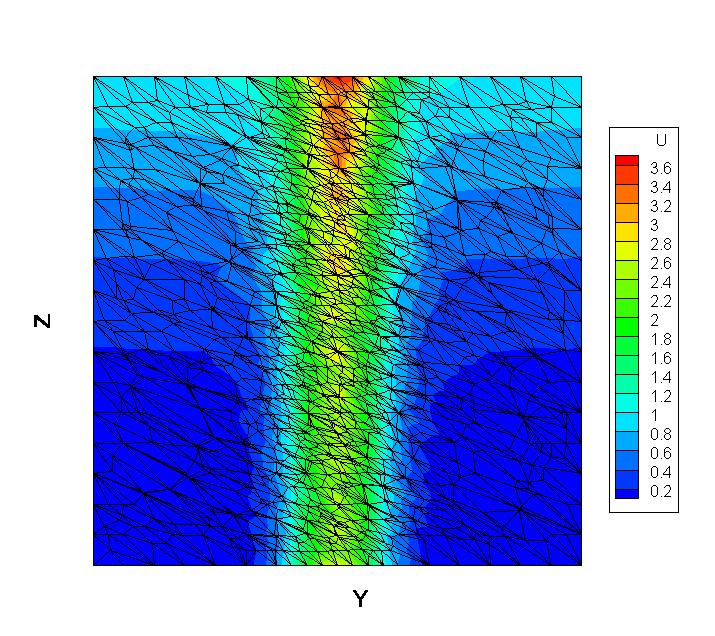}
\centerline{(d): $\M_{DMP+adap}$ mesh}
\end{minipage}
}
\caption{Example \ref{ex1}. Different meshes at cross-section $x=0.5$.}
\label{ex1-mesh2}
\end{figure}

\begin{table}[hbt]
\caption{Numerical results obtained with different meshes for Example \ref{ex1}.}
\vspace{2pt}
\centering
\begin{tabular}{c|c|cccccc}
\hline \hline
$\M_{id}$ & N & 384 & 3,072 & 24,576 & 196,608 & 750,000 & 1,572,864  \\
 mesh & $L^2$ error & 1.99e-1 & 1.07e-1 & 5.86e-2 & 1.98e-2 & 8.85e-3 & 5.55e-3 \\
& $u_{min}$ & -2.93e-2 & -4.06e-2 & -2.02e-2 & -4.79e-3 & -9.71e-4 & -3.87e-4 \\
\hline 
 $\M_{adap}$ & N & 572 & 2,865 & 17,304 & 133,012 & 627,329 & 1,521,648 \\
 mesh & $L^2$ error & 1.20e-1 & 5.67e-2 & 1.41e-2 & 3.72e-3 & 1.53e-3 & 9.56e-4 \\
& $u_{min}$ & -1.94e-2 & -2.26e-2 & 0 & -4.17e-4 & -5.68e-5 & -7.04e-5 \\
\hline
$\M_{DMP}$  & N & 1,254 & 5,962 & 29,023 & 201,039 & 736,646 & 1,542,910  \\
 mesh & $L^2$ error & 2.74e-1 & 1.96e-1 & 9.84e-2 & 3.40e-2 & 1.59e-2 & 1.03e-2  \\
& $u_{min}$ & -2.90e-1 & -1.28e-1 & -7.84e-2 & -2.52e-3 & 0 & 0  \\
\hline
$\M_{DMP+adap}$ & N & 1,017 & 3,500 & 20,099 & 132,215 & 477,802 & 978,845  \\
 mesh  & $L^2$ error & 2.71e-1 & 1.48e-1 & 5.12e-2 & 1.45e-2 & 6.34e-3 & 3.88e-3 \\
& $u_{min}$ & -6.23e-2 & -2.25e-2 & -5.60e-2 & 0 & 0 & 0 \\
\hline \hline
\end{tabular}
\label{ex1-result}
\end{table}

\begin{figure}[hbt]
\centering
\includegraphics[scale=0.5]{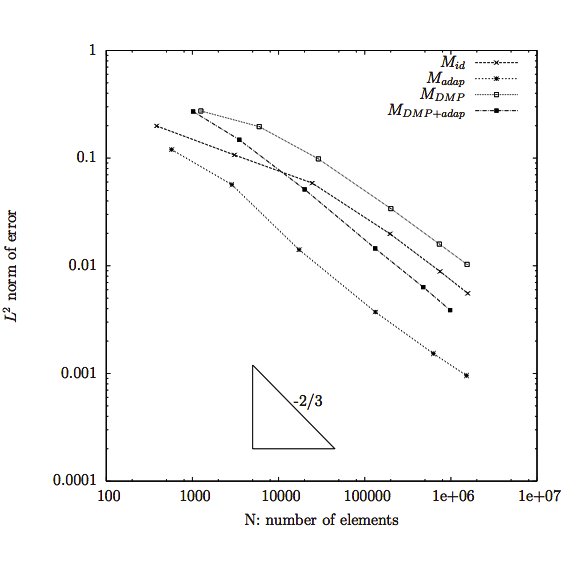}
\caption{Example \ref{ex1}. Convergence history of $L^2$-norm of the error for different meshes.}
\label{ex1-conv}
\end{figure}

\end{exam}

\begin{exam}
\label{ex2}

For this example, we would like to discuss the relation between mesh quality and MP satisfaction for three different cases.
The domain is $(0,1)^3$ with a cubic hole $[0.4,0.6]^3$ cut inside.
The problem is in the form (\ref{bvp-pde}) with $f = 0$ and the Dirichlet boundary condition $u = 0$
on the outer surface and $u=4$ on the inner surface. 
The diffusion matrix is in the form of (\ref{df-eigen}) with $\phi=\pi/4$, $k_1=100$, $k_2=10$, and $k_3=10$.
The primary diffusion direction is on the plane of $y-z=d$ ($-1 \le d \le 1$) and
in the direction defined by the angle $\theta$. 

We first consider two different values of $\theta$ and then a diffusion matrix with a strong
anisotropic feature. The continuous problem satisfies MP and the exact solution
satisfies $0 \le u \le 4$. 

\vspace{10pt}

\textbf{ Case 1.} $\theta=\pi/4$. Fig.~\ref{ex2-soln} shows the numerical solution obtained using
a fine $\M_{id}$ mesh of 5,952,000 tetrahedra. Four different meshes are shown in Fig.~\ref{ex2-mesh}. 
Table~\ref{ex2-result} shows the minimal value $u_{min}$ in the numerical solution obtained using different meshes.
The corresponding mesh quality measures $||Q_{ali}||_{L^2}$, $||Q_{ali}||_{L^{\infty}}$, and $||Q_{eq}||_{L^2}$ are also reported in Table~\ref{ex2-result}. Recall that only $Q_{ali}$ is related to MP satisfaction, while $Q_{eq}$
indicates how closely the mesh satisfies the equidistribution condition \eqref{eq-1}.
The mesh quality measures for $\M_{id}$ mesh are calculated
using $\M = \D^{-1}$ in order to see if it satisfies (\ref{meas-ali3}).
The numerical solutions from all the meshes have the maximal value of $u_{max}=4$. 

It can be seen that the numerical solution obtained using the $\M_{adap}$ mesh has
undershoots ($u_{min}<0$) while the solutions obtained using other meshes satisfy MP.
Interestingly, for this case MMG3D produces the same mesh from $\M_{DMP}$ and $\M_{id}$.
An explanation is that the initial mesh, the $\M_{id}$ mesh, is already very close to an $\M_{DMP}$ mesh
(with $||Q_{ali}||_{L^2}=2.32$). Since a half of the $\M_{id}$ mesh elements are aligned with the primary diffusion direction and the anisotropy is not significant, the elements do not need to be stretched more. In fact, the $\M_{id}$ mesh is closer to an $M_{DMP}$ mesh when $k_1=50$ with $||Q_{ali}||_{L^2}=1.72$. When the anisotropy is more significant,
the $\M_{id}$ mesh is farther away from the $\M_{DMP}$ mesh and the elements need to be stretched more along the primary diffusion direction. MMG3D will then adapt the mesh to be more different from the $\M_{id}$ mesh,
as will be shown in Case 3.
Moreover, both $\M_{adap}$ and $\M_{DMP+adap}$ meshes have elements of worse shape
(with large $||Q_{ali}||_{L^{\infty}}$)
but the overall quality is acceptable (with relatively small $||Q_{ali}||_{L^2}$).

\begin{figure}[hbt]
\centering
\includegraphics[scale=0.25]{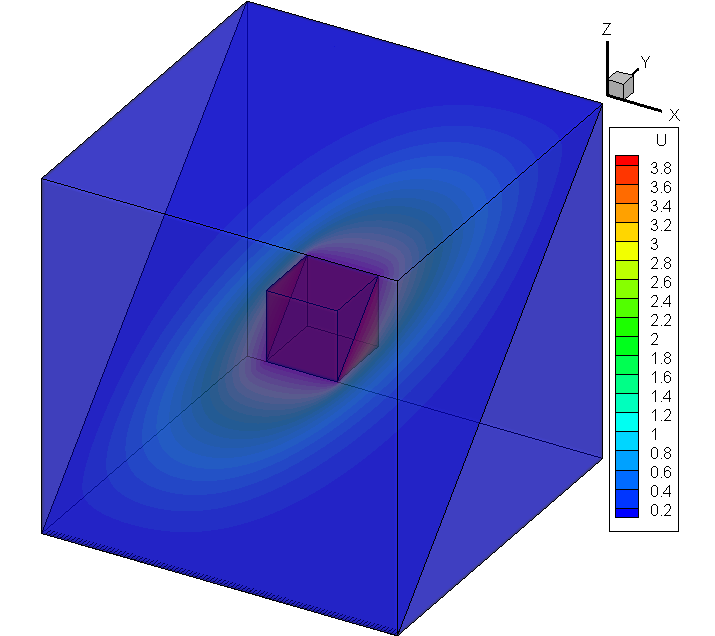}
\caption{Example \ref{ex2} Case 1. Numerical solution in the domain $\Omega$ and on the cross-section $z=y$.}
\label{ex2-soln}
\end{figure}

\begin{figure}[thb]
\centering
\hbox{
\hspace{1mm}
\begin{minipage}[b]{2.5in}
\includegraphics[scale=0.25]{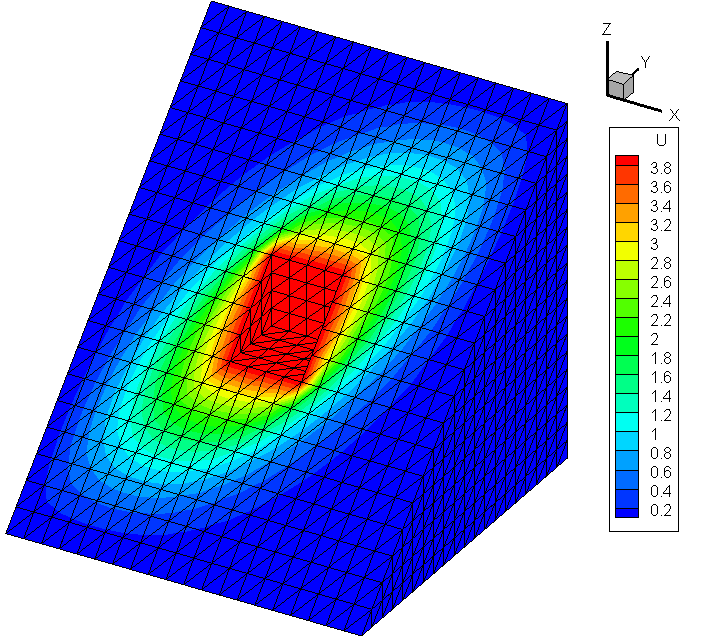}
\centerline{(a): $\M_{id}$ mesh, $N=47,616$}
\end{minipage}
\hspace{1mm}
\begin{minipage}[b]{2.5in}
\includegraphics[scale=0.25]{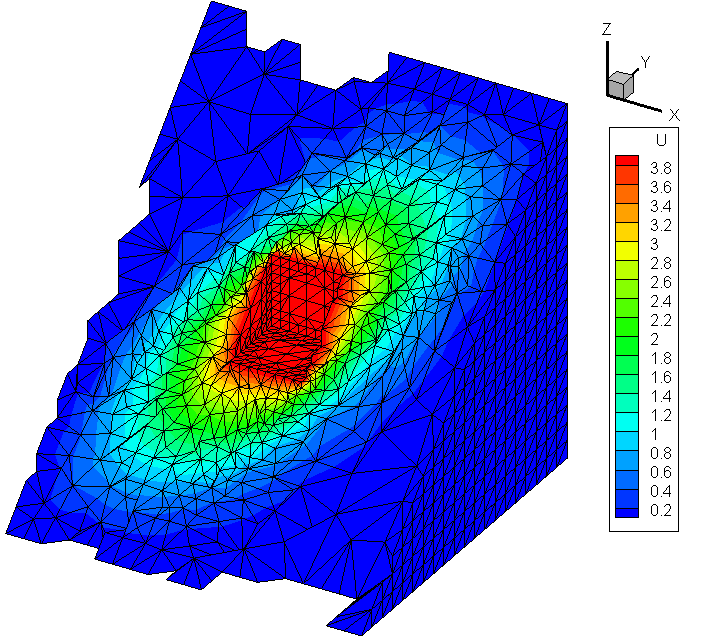}
\centerline{(b): $\M_{adap}$ mesh, $N=36,066$}
\end{minipage}
}
\vspace{5mm}
\hbox{
\hspace{1mm}
\begin{minipage}[b]{2.5in}
\includegraphics[scale=0.25]{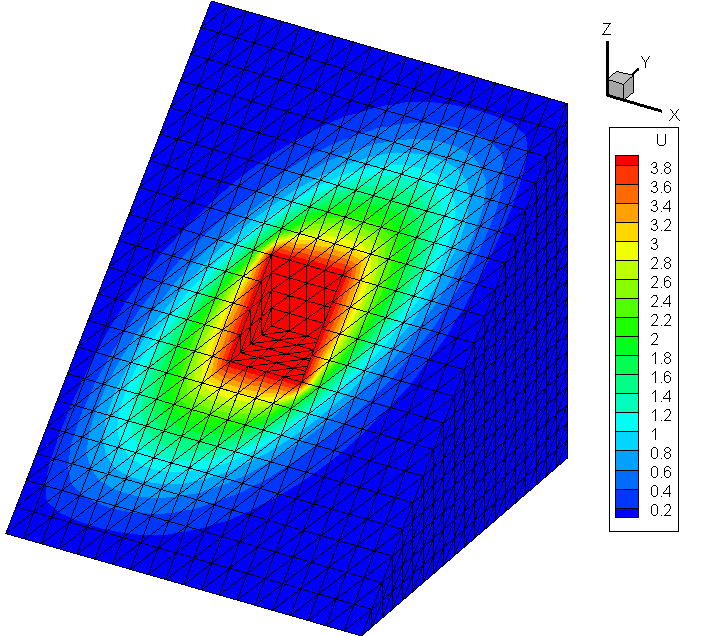}
\centerline{(c): $\M_{DMP}$ mesh, $N=47,616$}
\end{minipage}
\hspace{1mm}
\begin{minipage}[b]{2.5in}
\includegraphics[scale=0.25]{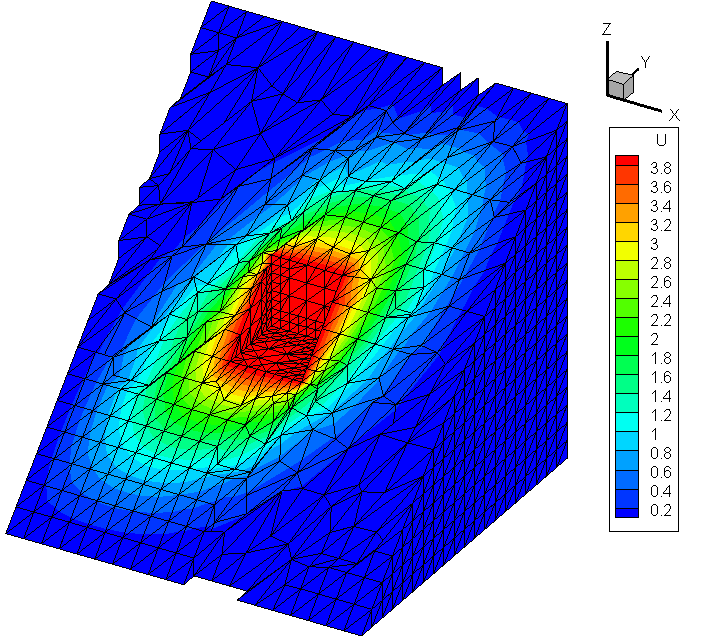}
\centerline{(d): $\M_{DMP+adap}$ mesh, $N=35,528$}
\end{minipage}
}
\caption{Example \ref{ex2} Case 1. Different meshes at cross-section $z=y$ for $\theta=\pi/4$.}
\label{ex2-mesh}
\end{figure}

\begin{table}[hbt]
\caption{Minimal solution values and mesh quality measures using different meshes for Example \ref{ex2} Case 1.}
\vspace{2pt}
\centering
\begin{tabular}{c|r|l|c|r|c}
\hline \hline
 Mesh & $N$ & $u_{min}$ & $||Q_{ali}||_{L^2}$ & $||Q_{ali}||_{L^{\infty}}$ & $||Q_{eq}||_{L^2}$ \\
\hline
 & 744 & 0 & 2.32 & 2.58 & 1.00  \\
$\M_{id}$ & 5,952 & 0 & 2.32 & 2.58 & 1.00 \\
 mesh & 47,616 & 0 & 2.32 & 2.58 & 1.00 \\
& 380,928 & 0 & 2.32 & 2.58 & 1.00 \\
& 3,047,424 & 0 & 2.32 & 2.58 & 1.00 \\
\hline 
 & 976 & 0 & 3.22 & 13.0 & 1.15  \\
 $\M_{adap}$ & 5,616 & -1.47e-4 & 3.98 & 22.17 & 1.06 \\
mesh & 36,066 & -2.38e-3 & 4.38 & 47.11 & 1.11 \\
& 278,023 & -1.30e-3 & 4.25 & 106.7 & 1.16 \\
& 2,779,767 & 0 & 4.21 & 151.5 & 1.38 \\
\hline
 & 744 & 0 & 2.32 & 2.58 & 1.00  \\
$\M_{DMP}$ & 5,952 & 0 & 2.32 & 2.58 & 1.00 \\
 mesh & 47,616 & 0 & 2.32 & 2.58 & 1.00 \\
& 380,928 & 0 & 2.32 & 2.58 & 1.00 \\
& 3,047,424 & 0 & 2.32 & 2.58 & 1.00 \\
\hline
 & 748 & 0 & 2.28 & 3.36 & 1.08  \\
$\M_{DMP+adap}$ & 6,539 & 0 & 2.32 & 7.41 & 1.07 \\
 mesh & 35,528 & 0 & 2.27 & 9.01 & 1.17 \\
& 397,125 & 0 & 2.29 & 8.14 & 1.12 \\
& 3,361,403 & 0 & 2.32 & 12.43 & 1.05 \\
\hline \hline
\end{tabular}
\label{ex2-result}
\end{table}

\vspace{10pt}

\textbf{ Case 2.} $\theta=3\pi/4$. In this case, the elements of the $\M_{id}$ mesh are no longer aligned
along the primary diffusion direction. Meshes on the cross-section of $z=y$ are shown in Fig.~\ref{ex2-mesh2}.
Notice that the orientation of elements in the $\M_{DMP}$ mesh ($3\pi/4$ direction) is very different from that
in the $\M_{id}$ mesh ($\pi/4$ direction). The results of $u_{min}$ and mesh quality measures are listed
in Table~\ref{ex2-result2}.

One can see that none of the meshes satisfies (\ref{meas-ali3}), however, recall that (\ref{meas-ali3}) or (\ref{meshcnd-5}) is only a sufficient condition for the MP satisfaction.
Notice that the numerical solutions obtained from both the $\M_{DMP}$ mesh and $M_{{DMP+adap}}$ mesh satisfy MP whereas others do not.
Fig.~\ref{ex2-mesh2} also shows that $M_{DMP+adap}$ mesh not only tempts to make elements to be aligned
with the primary diffusion direction but also concentrates elements to minimize the solution error.

\begin{figure}[hbt]
\centering
\hbox{
\hspace{1mm}
\begin{minipage}[b]{2.5in}
\includegraphics[scale=0.25]{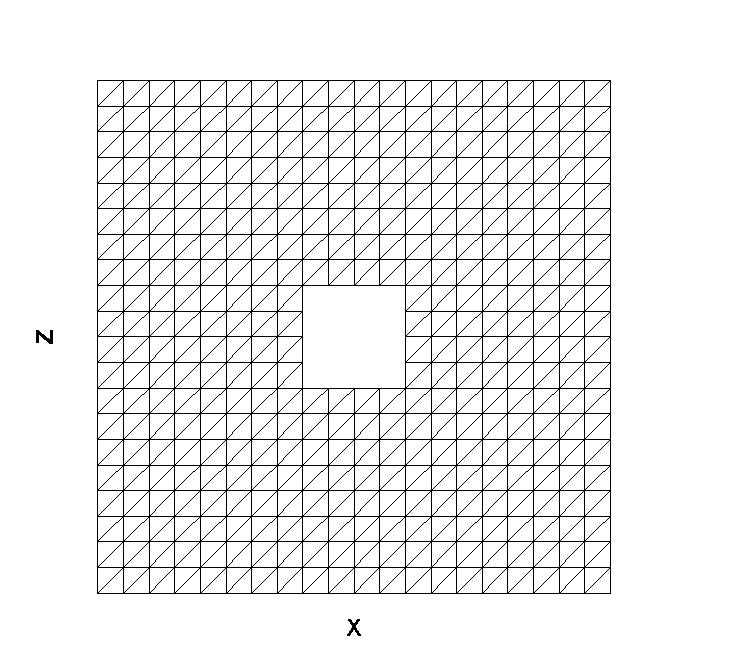}
\centerline{(a): $\M_{id}$ mesh, $N=47,616$}
\end{minipage}
\hspace{1mm}
\begin{minipage}[b]{2.5in}
\includegraphics[scale=0.25]{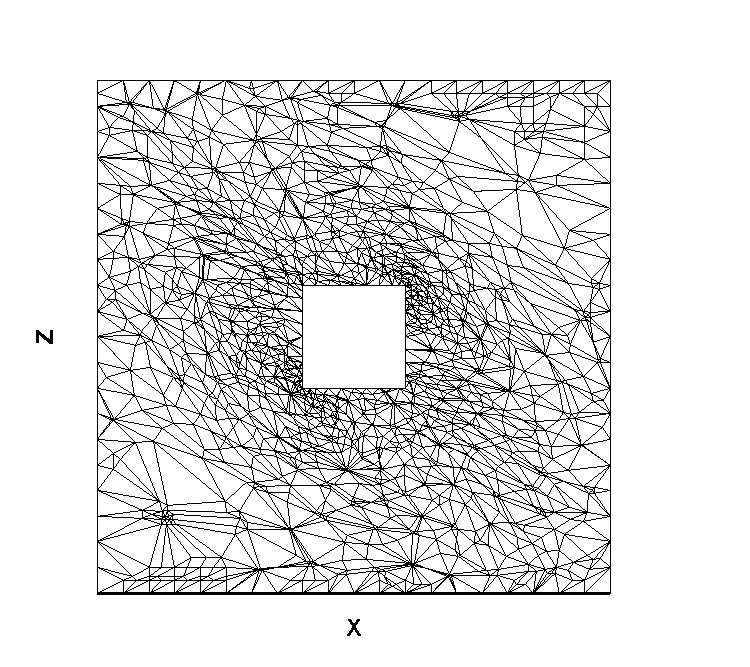}
\centerline{(b): $\M_{adap}$ mesh, $N=37,899$}
\end{minipage}
}
\hbox{
\hspace{1mm}
\begin{minipage}[b]{2.5in}
\includegraphics[scale=0.25]{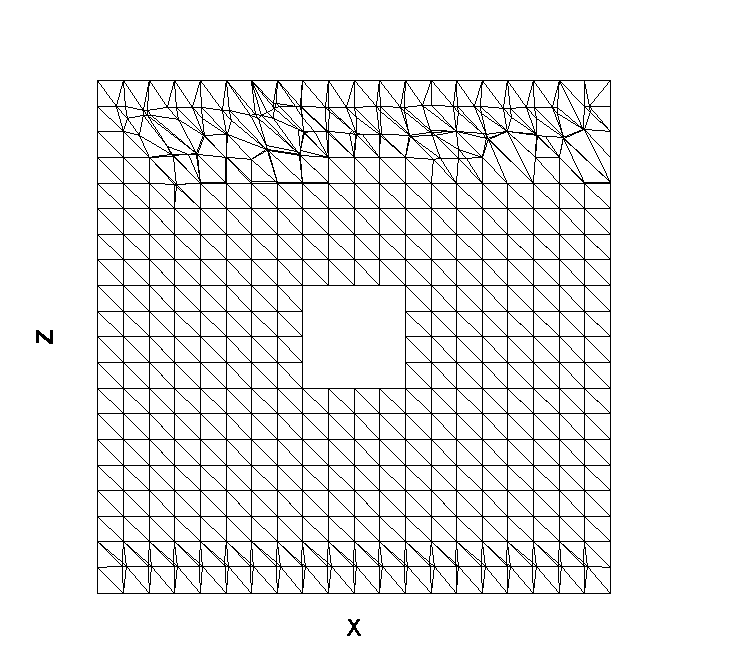}
\centerline{(c): $\M_{DMP}$ mesh, $N=47,454$}
\end{minipage}
\hspace{1mm}
\begin{minipage}[b]{2.5in}
\includegraphics[scale=0.25]{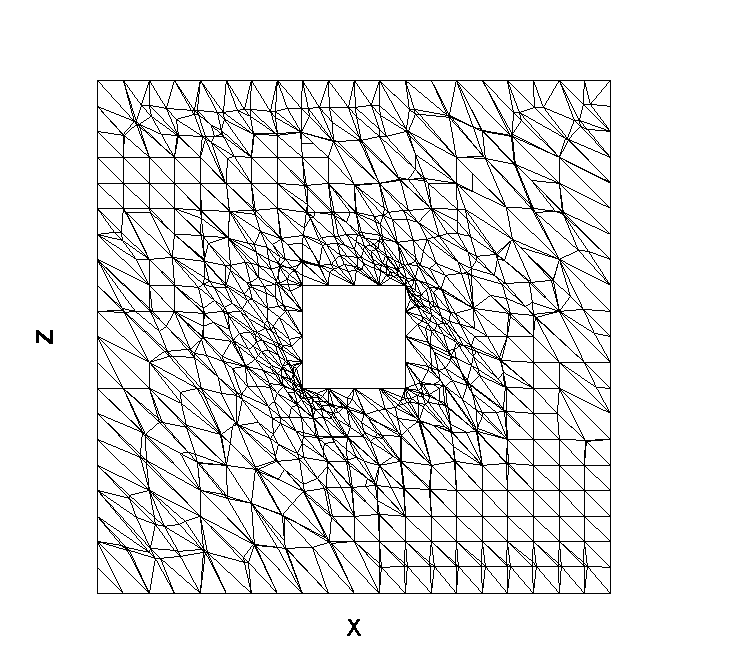}
\centerline{(d): $\M_{DMP+adap}$ mesh, $N=48693$}
\end{minipage}
}
\caption{Example \ref{ex2} Case 2. Planar view of meshes at cross-section $z=y$ for $\theta=3\pi/4$.
Notice that a planar cut of a 3D mesh does not necessarily form a 2D mesh.}
\label{ex2-mesh2}
\end{figure}

\begin{table}[hbt]
\caption{Minimal solution values and mesh quality measures using different meshes for Example \ref{ex2} Case 2.}
\vspace{2pt}
\centering
\begin{tabular}{c|r|l|c|r|c}
\hline \hline
 Mesh & $N$ & $u_{min}$ & $||Q_{ali}||_{L^2}$ & $||Q_{ali}||_{L^{\infty}}$ & $||Q_{eq}||_{L^2}$ \\
\hline
 & 744 & 0 & 7.15 & 9.42 & 1.00  \\
$\M_{id}$ & 5,952 & -4.82e-3 & 7.15 & 9.42 & 1.00 \\
 mesh & 47,616 & -3.38e-4 & 7.15 & 9.42 & 1.00 \\
& 380,928 & -1.16e-5 & 7.15 & 9.42 & 1.00 \\
& 3,047,424 & -5.10e-7 & 7.15 & 9.42 & 1.00 \\
\hline 
 & 834 & -1.26e-1 & 3.87 & 22.66 & 1.24  \\
 $\M_{adap}$ & 5,485 & -3.10e-3 & 4.99 & 42.44 & 1.34 \\
mesh & 37,899 & -2.46e-3 & 4.42 & 35.13 & 1.29 \\
& 314,228 & -7.62e-4 & 12.1 & 291.4 & 1.10 \\
& 2,659,107 & -8.26e-5 & 13.2 & 491.3 & 1.18 \\
\hline
 & 834 & 0 & 3.46 & 14.5 & 1.08  \\
$\M_{DMP}$ & 6,052 & 0 & 2.43 & 15.08 & 1.03 \\
 mesh & 47,454 & 0 & 2.08 & 15.3 & 1.02 \\
& 378,787 & 0 & 1.93 & 22.2 & 1.01 \\
& 3,035,251 & 0 & 1.87 & 24.0 & 1.00 \\
\hline
 & 744 & 0 & 3.46 & 9.42 & 1.10  \\
$\M_{DMP+adap}$ & 6,284 & 0 & 2.58 & 9.42 & 1.16 \\
 mesh & 48,693 & 0 & 2.32 & 38.09 & 1.11 \\
& 428,732 & 0 & 2.09 & 167.9 & 1.10 \\
& 3,576,978 & 0 & 1.94 & 5790 & 1.07 \\
\hline \hline
\end{tabular}
\label{ex2-result2}
\end{table}

\vspace{10pt}

\textbf{ Case 3.} $\theta=\pi/4$, $k_1=1000$, $k_2=1$, $k_3=1$. In this case, the diffusion is much faster
in the direction of $\theta=\pi/4$ than in other directions. Fig.~\ref{ex2-soln3} shows the numerical solution using a fine $M_{id}$ mesh with 5,952,000 tetrahedra. A planar view of four different meshes at cross-section $z=y$ is shown
in Fig.~\ref{ex2-mesh3}. The results of $u_{min}$ and mesh quality measures are reported in Table~\ref{ex2-result3}.

Overall, all but the $\M_{id}$ mesh are close to being $\M$-uniform, with relatively small $\|Q_{ali}\|_{L^2}$
and  $\|Q_{eq}\|_{L^2}$. But they do not satisfy (\ref{meshcnd-5}) so there is no guarantee that the solutions
will satisfy MP. Indeed, the numerical solutions obtained from all four meshes violate MP except
for fine $\M_{DMP}$ and $\M_{DMP+adap}$ meshes. As can be seen from Table \ref{ex2-result3}, 
$\M_{DMP}$ and $M_{DMP+adap}$ meshes improve the MP satisfaction of the numerical solution.
For the numerical solution obtained using the $\M_{id}$ mesh with $N=380,928$, it has the minimal value
$u_{min}=-7.63 \times 10^{-2}$, and for the $\M_{adap}$ mesh with $N=315,947$, $u_{min}=-5.79 \times 10^{-3}$.
On the other hand, the numerical solutions obtained using the $\M_{DMP}$ mesh with $N=310,147$
and the $\M_{DMP+adap}$ mesh with $N=396,625$ already satisfy MP. 

\end{exam}

\begin{figure}[hbt]
\centering
\includegraphics[scale=0.25]{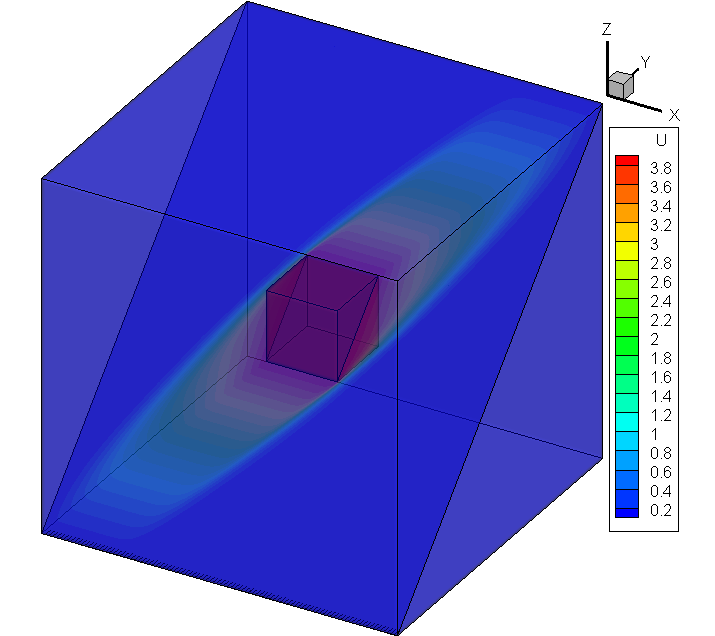}
\caption{Example \ref{ex2} Case 3. Numerical solution in the domain $\Omega$ and on the cross-section $z=y$.}
\label{ex2-soln3}
\end{figure}

\begin{figure}[hbt]
\centering
\hbox{
\hspace{1mm}
\begin{minipage}[b]{2.5in}
\includegraphics[scale=0.25]{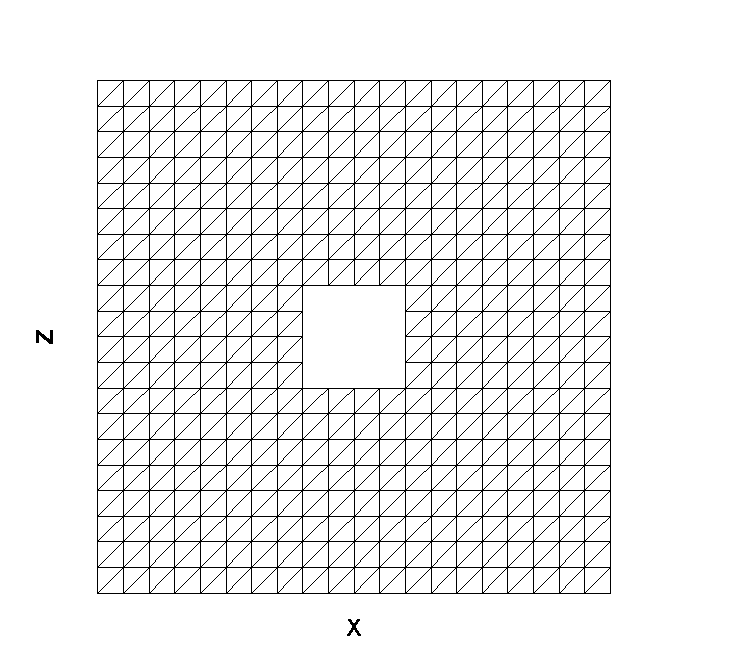}
\centerline{(a): $\M_{id}$ mesh, $N=47,616$}
\end{minipage}
\hspace{1mm}
\begin{minipage}[b]{2.5in}
\includegraphics[scale=0.25]{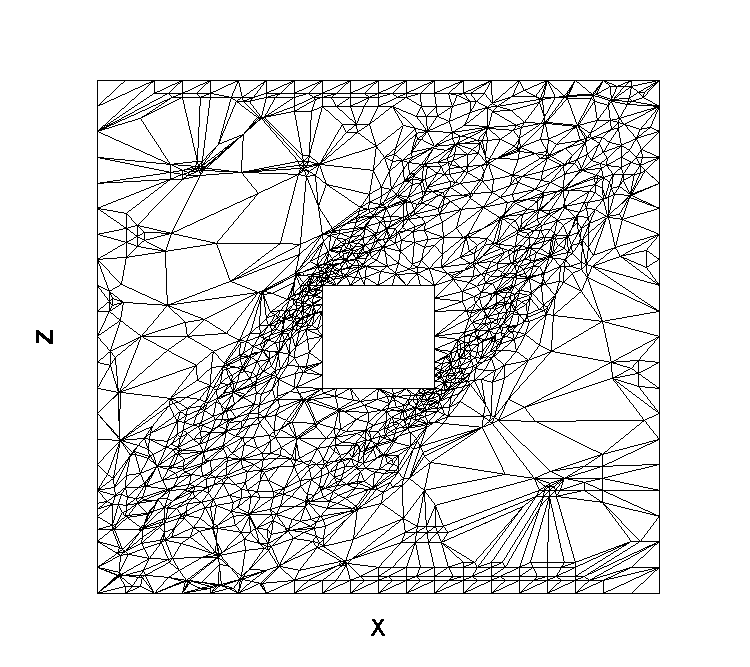}
\centerline{(b): $\M_{adap}$ mesh, $N=46,334$}
\end{minipage}
}
\hbox{
\hspace{1mm}
\begin{minipage}[b]{2.5in}
\includegraphics[scale=0.25]{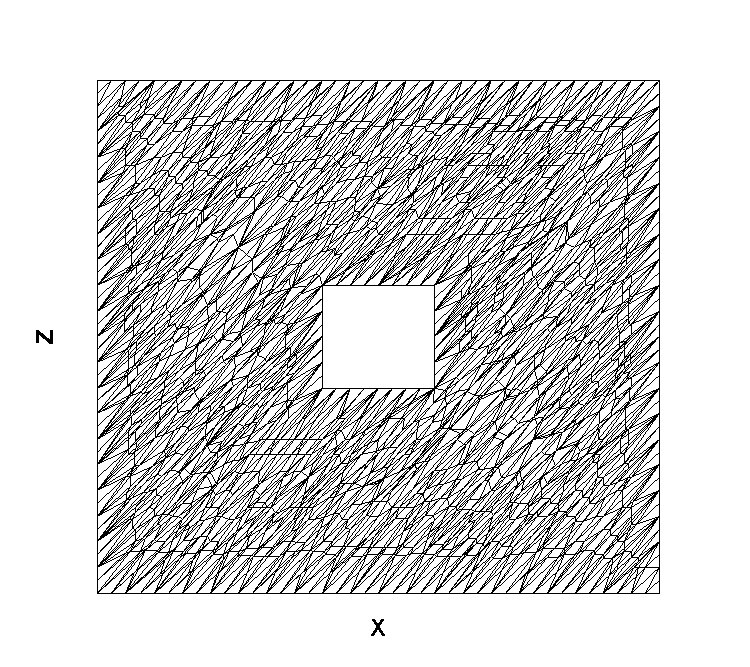}
\centerline{(c): $\M_{DMP}$ mesh, $N=48,925$}
\end{minipage}
\hspace{1mm}
\begin{minipage}[b]{2.5in}
\includegraphics[scale=0.25]{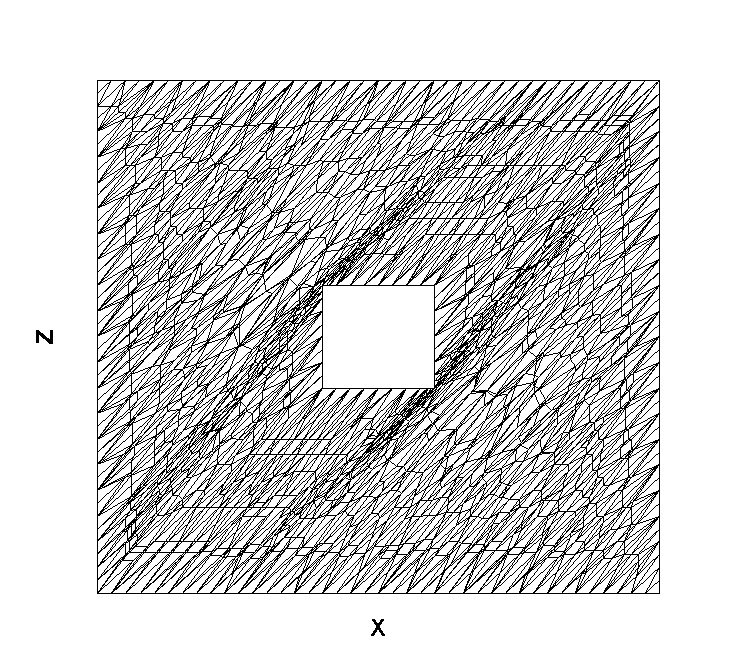}
\centerline{(d): $\M_{DMP+adap}$ mesh, $N=48,267$}
\end{minipage}
}
\caption{Example \ref{ex2} Case 3. Planar view of meshes at cross-section $z=y$
for $\theta=\pi/4$, $k_1=1000$, $k_2=1$, and $k_3=1$.
Notice that a planar cut of a 3D mesh does not necessarily form a 2D mesh.
}
\label{ex2-mesh3}
\end{figure}

\begin{table}[hbt]
\caption{Minimal solution values and mesh qualities using different meshes for Example \ref{ex2} Case 3.}
\vspace{2pt}
\centering
\begin{tabular}{c|r|l|c|r|c}
\hline \hline
 Mesh & $N$ & $u_{min}$ & $||Q_{ali}||_{L^2}$ & $||Q_{ali}||_{L^{\infty}}$ & $||Q_{eq}||_{L^2}$ \\
\hline
 & 744 & -6.15e-2 & 46.04 & 49.99 & 1.00  \\
$\M_{id}$ & 5,952 & -1.09e-1 & 46.04 & 49.99 & 1.00 \\
 mesh & 47,616 & -1.14e-1 & 46.04 & 49.99 & 1.00 \\
& 380,928 & -7.63e-2 & 46.04 & 49.99 & 1.00 \\
& 3,047,424 & -3.78e-2 & 46.04 & 49.99 & 1.00 \\
\hline 
 & 924 & -8.04e-2 & 3.72 & 23.43 & 1.29  \\
 $\M_{adap}$ & 5,137 & -3.38e-2 & 4.37 & 24.22 & 1.18 \\
mesh & 46,334 & -2.08e-2 & 5.00 & 50.52 & 1.32 \\
& 315,947 & -5.79e-3 & 4.65 & 73.36 & 1.42 \\
& 3,292,289 & -1.53e-3 & 3.67 & 135.3 & 1.15 \\
\hline
 & 2,995 & -1.35e-1 & 12.05 & 49.99 & 1.32  \\
$\M_{DMP}$ & 6,203 & -8.45e-2 & 6.88 & 94.55 & 1.55 \\
 mesh & 48,925 & -5.03e-9 & 3.86 & 110.9 & 1.38 \\
& 310,147 & 0 & 2.56 & 49.99 & 1.19 \\
& 2,064,353 & 0 & 2.42 & 49.99 & 1.20 \\
\hline
 & 1,898 & 0 & 10.34 & 49.99 & 1.74  \\
$\M_{DMP+adap}$ & 6,919 & -2.67e-3 & 6.90 & 54.49 & 1.36 \\
 mesh & 48,267 & -1.01e-7 & 3.77 & 63.01 & 1.28 \\
& 396,625 & 0 & 2.78 & 78.36 & 1.22 \\
& 2,716,462 & 0 & 2.43 & 54.49 & 1.48 \\
\hline \hline
\end{tabular}
\label{ex2-result3}
\end{table}

\section{Application to fractured reservoir simulation in petroleum engineering}
\label{sec-app}

Numerical simulation plays an important role in petroleum engineering to predict production rate,
optimize hydraulic fracturing design, and evaluate enhanced oil recovery processes.
In those computations, the mesh has to be sufficiently refined around wellbore or fractures
to accurately represent the flow effects thereon because the fractures are significantly smaller but with much higher permeability comparing to the reservoir matrix. It is challenging to consider the mesh inside a fracture due to its small scale.

In recent years, shale gas reservoirs have become an attractive source for natural gas production largely due to the advancement of horizontal drilling and hydraulic fracturing techniques \cite{WCJR14}.
Shale reservoirs typically have extremely low permeability. For example, the representative
permeability is $1.0 \times 10^{-4}$ millidarcies (mD) in the Barnett shale but is $5.0 \times 10^4$ mD
in fractures \cite{OFMB12}. The high permeability in fractures makes it possible to produce the gas
from the shale while it makes the reservoir highly heterogeneous and anisotropic.
The complexity of fracture network together with the complexity of shale pore structure
makes shale-gas reservoir simulation a very challenging task. Some researchers
focus on fluid flow models in nano-scale shale pores \cite{CLER10, LL14, MJL14, SK12}
and some other focus on different models of fracture network in the reservoir
\cite{CLET09, FMIB13, OFMB12, Rubin10}. 

As our first exploration of the use of anisotropic mesh adaptation in fractured reservoir simulation,
we consider the steady-state fluid flow that can be applied to effective permeability calculation \cite{BMT03,VPS13}. For simplicity, we consider incompressible single-phase fluid flow and choose the single porosity dual permeability model. Darcy's Law is chosen to describe the flow in both matrix and fractures with different effective permeability in the corresponding regions. We also assume that permeability does not change with pressure. Under this setting, we can perform mesh adaptation in both the matrix and the fractures.

We consider a reservoir with 3000 ft in $x$-direction, 1000 ft in $y$-direction, and 300 ft in $z$-direction,
and denote the domain as $\Omega$. Fig.~\ref{ex3-domain} shows the sketch of the half-panel of the reservoir
($0 \le y \le 500$). A horizontal well is located in the center of the reservoir with length $L_{w}=1400$ ft
along $x$-direction. The radius of the wellbore is 0.3 ft. The reservoir pressure is $P_{r}=3800$ psi and
the pressure in the wellbore is $P_{w}=1000$ psi. There are three fractures with half-length $L_{f}=400$ ft
and width $W_{f}=0.01$ ft at the location $x=800$ ft, $x=1500$ ft, and $x=2200$ ft, respectively. 
The angle between the fractures and the positive $x$-axis is $\pi/4$. 
The height of the fractures is the same as the height of the reservoir. The subdomain formed by
the fractures is denoted as $\Omega_f$. The permeability in the matrix is $k_{mpar}=0.1$ mD in the direction 
parallel to the fractures, is $k_{mperp}=5$ mD in the direction perpendicular to the the fractures on $xy$-plane, and is $k_{mpar}$ in the direction of $z$-axis. 
The permeability in the fractures is $k_{pf}=10^4$ mD along the fractures and is considered the same as that permeability in the matrix
in other directions. For successful computation, a common technique is to scale up the width of the fracture while keeping the same conductivity, that is, $K_{pf}W_{f}=100$ mD-ft for this example. In our computations, the width of each fracture is scaled up to $W_{f}=10$ ft and the permeability is reduced to $K_{pf}=10$ mD. Outside of the fractures, the primary diffusion direction in the matrix is perpendicular to the fractures.

\begin{figure}[htb]
\centering
\begin{tikzpicture}[scale=0.3]
\draw[->] (0.,0.) -- (33.,0.);
\draw[->] (0.,0.) -- (0.,12.);
\draw[->] [dash pattern=on 8pt off 8pt] (0.,0.) -- (13.,13.);
\draw (0.,8.)-- (30.,8.);
\draw (30.,8.)-- (40.,18.);
\draw (40.,18.)-- (10.,18.);
\draw (10.,18.)-- (0.,8.);
\draw (0.,0.)-- (30.,0.);
\draw (30.,0.)-- (40.,10.);
\draw [dash pattern=on 8pt off 8pt] (40.,10.)-- (10.,10.);
\draw [dash pattern=on 8pt off 8pt] (10.,10.)-- (0.,0.);
\draw (0.,8.)-- (0.,0.);
\draw (30.,8.)-- (30.,0.);
\draw (40.,18.)-- (40.,10.);
\draw [dash pattern=on 8pt off 8pt] (10.,18.)-- (10.,10.);
\draw (7.5,0.)-- (7.5,8.);
\draw (7.5,8.)-- (15.5,13.);
\draw [dash pattern=on 8pt off 8pt] (15.5,13.)-- (15.5,5.);
\draw [dash pattern=on 8pt off 8pt] (15.5,5.)-- (7.5,0.);
\draw (8.5,0.)-- (8.5,8.);
\draw (8.5,8.)-- (16.5,13.);
\draw [dash pattern=on 8pt off 8pt] (16.5,13.)-- (16.5,5.);
\draw [dash pattern=on 8pt off 8pt] (16.5,5.)-- (8.5,0.);
\draw (15.5,13.)-- (16.5,13.);
\draw [dash pattern=on 8pt off 8pt] (15.5,5.)-- (16.5,5.);
\draw (14.5,0.)-- (14.5,8.);
\draw (14.5,8.)-- (22.5,13.);
\draw [dash pattern=on 8pt off 8pt] (22.5,13.)-- (22.5,5.);
\draw [dash pattern=on 8pt off 8pt] (22.5,5.)-- (14.5,0.);
\draw (15.5,0.)-- (15.5,8.);
\draw (15.5,8.)-- (23.5,13.);
\draw [dash pattern=on 8pt off 8pt] (23.5,13.)-- (23.5,5.);
\draw [dash pattern=on 8pt off 8pt] (23.5,5.)-- (15.5,0.);
\draw (22.5,13.)-- (23.5,13.);
\draw [dash pattern=on 8pt off 8pt] (22.5,5.)-- (23.5,5.);
\draw (21.5,0.)-- (21.5,8.);
\draw (21.5,8.)-- (29.5,13.);
\draw [dash pattern=on 8pt off 8pt] (29.5,13.)-- (29.5,5.);
\draw [dash pattern=on 8pt off 8pt] (29.5,5.)-- (21.5,0.);
\draw (22.5,0.)-- (22.5,8.);
\draw (22.5,8.)-- (30.5,13.);
\draw [dash pattern=on 8pt off 8pt] (30.5,13.)-- (30.5,5.);
\draw [dash pattern=on 8pt off 8pt] (30.5,5.)-- (22.5,0.);
\draw (29.5,13.)-- (30.5,13.);
\draw [dash pattern=on 8pt off 8pt] (29.5,5.)-- (30.5,5.);
\draw (33,0) node[right] {$x$};
\draw (0,12) node[left] {$z$};
\draw (13,13) node[left] {$y$};
\draw (0,0) node[below] {$0$};
\draw (30,0) node[below] {$3000$};
\draw (0,8) node[left] {$300$};
\draw (9.5,10) node[left] {$500$};
\draw (8,0) node[below] {$800$};
\draw (15,0) node[below] {$150$};
\draw (22,0) node[below] {$2200$};
\draw (5,10) node {$\Gamma_1$};
\draw (32,12) node {$\Gamma_1$};
\draw (35,8) node {$\Gamma_1$};
\draw (6,2) node {$\Gamma_2$};
\draw (13,1) node {$\Gamma_2$};
\draw (20,1) node {$\Gamma_2$};
\draw (28,2) node {$\Gamma_2$};
\draw (19,7) node[right] {$\Gamma_3$};
\draw (18,14) node[right] {$\Gamma_3$};
\draw (7.2,4) node[right] {$\Gamma_4$};
\draw (14.2,4) node[right] {$\Gamma_4$};
\draw (21.2,4) node[right] {$\Gamma_4$};
\end{tikzpicture}
\caption{Fractured reservoir simulation. The sketch of the half-panel of the reservoir where the angle between the fractures and the $x$-axis is $\pi/4$. $\Gamma_1$ consists of left side ($x=0$), right side ($x=3000$), and back side ($y=500$). $\Gamma_2$ consists of front side ($y=0$). $\Gamma_3$ consists of the bottom side ($z=0$) and top side ($z=300$). $\Gamma_4$ consists of the faces of the fractures on front side.}
\label{ex3-domain}
\end{figure}

The faces on the left side ($x=0$), right side ($x=3000$), and back side ($y=500$) are denoted as $\Gamma_1$.
$\Gamma_2$ denotes the front side ($y=0$), and $\Gamma_3$ denotes the bottom side ($z=0$) and
top side ($z=300$). $\Gamma_4$ denotes the faces of the fractures on the front side $\Gamma_2$.
The horizontal well is in the center of the face $\Gamma_2$ along the $x$-direction.
For simplicity, we assume that the pressure on the whole surface $\Gamma_4$ is the same as
the pressure in the wellbore $P_w=1000$ psi. The pressure on $\Gamma_1$ is always the same as
the reservoir pressure $P_r=3800$ psi. There is no flow through $\Gamma_2$ except in $\Gamma_4$, and
no flow through $\Gamma_3$ either. 
The mathematical formulation of the model is given by 
\beq
\label{ex-3}
\begin{cases}
 - \nabla \cdot (\mathbb{K} \, \nabla P)  =  0, & \quad \mbox{ in } \Omega \\
P = P_{r}, & \quad \mbox{ on } \Gamma_1 \\
P = P_{w}, & \quad \mbox{ on } \Gamma_4 \\
\frac{\partial P}{\partial \V{n}} = 0, & \quad \mbox{ on } (\Gamma_2 \backslash \Gamma_4) \cup \Gamma_3
\end{cases}
\eeq
with
\beq
\mathbb{K} = \begin{bmatrix} 7.5 & 2.5 & 0 \\ 2.5 & 7.5 & 0 \\ 0 & 0 & 0.1 \end{bmatrix} \mbox{ in } \Omega_f,\quad
\mbox{ and } \quad 
\mathbb{K} = \begin{bmatrix} 2.55 & -2.45 & 0 \\ -2.45 & 2.55 & 0 \\ 0 & 0 & 0.1 \end{bmatrix} \mbox{ in } \Omega \backslash \Omega_f, 
\eeq
where the viscosity of the fluid and the porosity of the matrix are taken out from the equation from the assumption that they are independent of pressure. Gravitational effects are also ignored. The numerical solution with a fine mesh of $2,804,175$ tetrahedra is shown in Fig.~\ref{ex3-soln}, where the mesh is manually refined around the fractures.

\begin{figure}[hbt]
\centering
\includegraphics[scale=0.4]{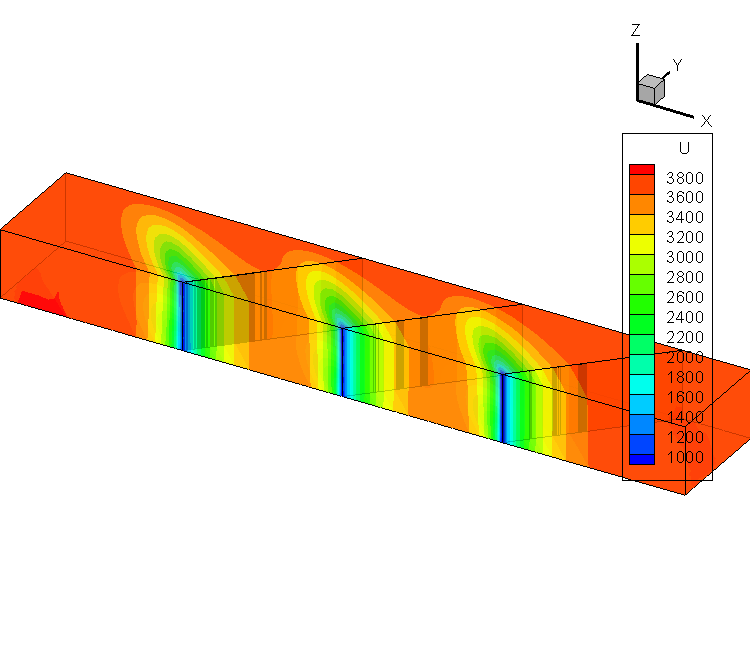}
\caption{Fractured reservoir simulation. Numerical solution in the domain $\Omega$ and on the cross-sections along the fractures at $x=800$, $x=1500$, and $x=2200$.}
\label{ex3-soln}
\end{figure}

The $xy$-planar view of four different meshes $\M_{id}$, $\M_{adap}$, $\M_{DMP}$, and $\M_{DMP+adap}$ at cross-section $z=100$ ft are shown
in Fig.~\ref{ex3-mesh}. In this example, the $\M_{id}$ mesh is the initial mesh that has manual local refinement around the fractures. The numerical solutions obtained from $\M_{id}$ and $\M_{adap}$ meshes have pressure values
larger than $3800$ psi inside the domain, which violates MP and causes spurious back flow as shown in
Fig.~\ref{ex3-soln-2}. The red shaded regions indicate unphysical pressure values.
The numerical solutions obtained using $\M_{DMP}$ and $\M_{DMP+adap}$ meshes both satisfy MP,
and the gradients of pressure are shown in Fig.~\ref{ex3-soln-2b}.

\begin{figure}[thb]
\centering
\hbox{
\hspace{1mm}
\begin{minipage}[b]{3in}
\includegraphics[scale=0.3]{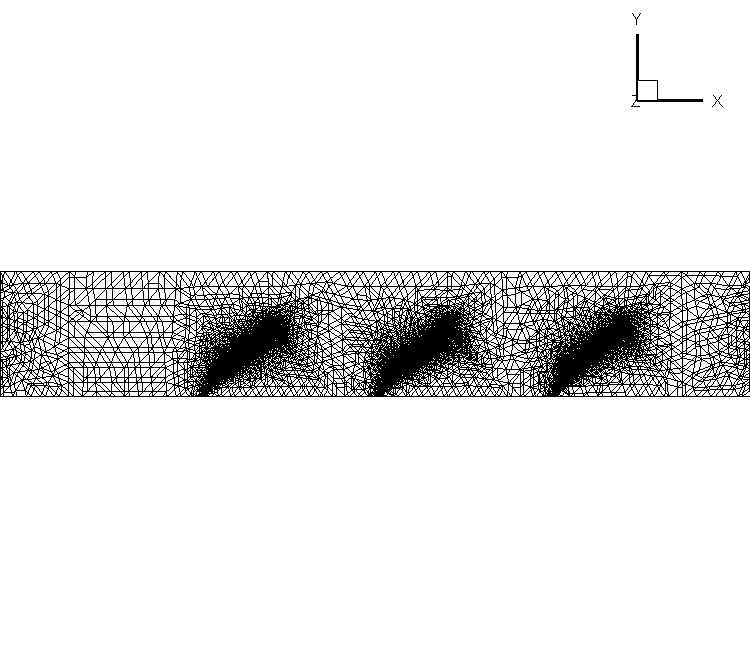}
\centerline{(a): $\M_{id}$ mesh, $N=347,760$}
\end{minipage}
\hspace{5mm}
\begin{minipage}[b]{3in}
\includegraphics[scale=0.3]{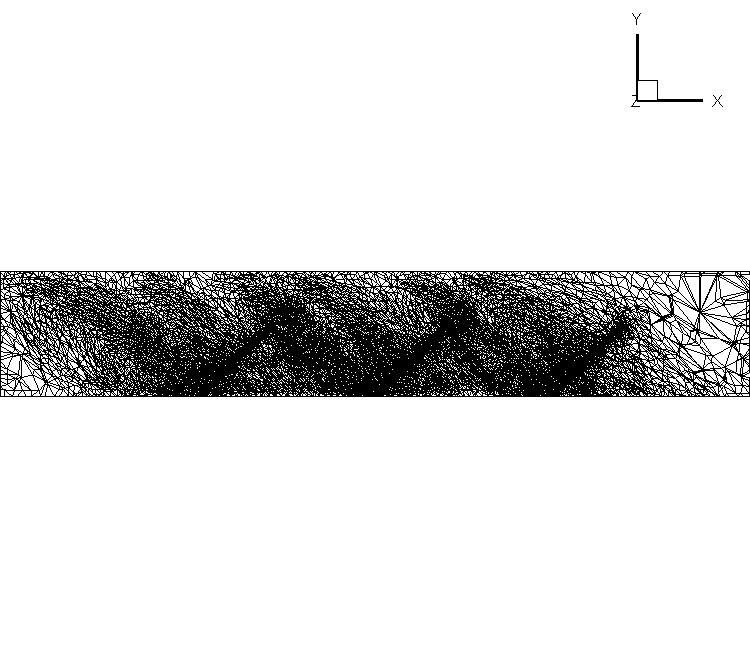}
\centerline{(b): $\M_{adap}$ mesh, $N=331,134$}
\end{minipage}
}
\vspace{5mm}
\hbox{
\hspace{1mm}
\begin{minipage}[b]{3in}
\includegraphics[scale=0.3]{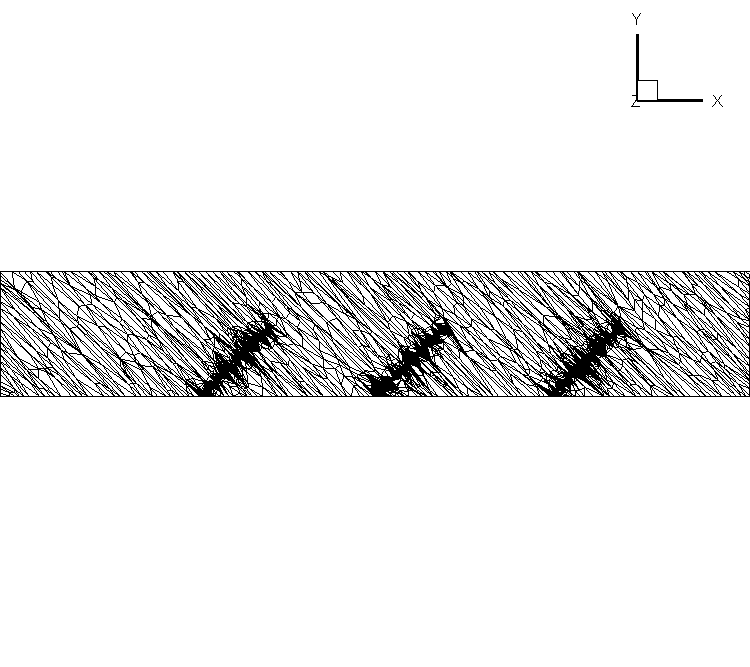}
\centerline{(c): $\M_{DMP}$ mesh, $N=351,206$}
\end{minipage}
\hspace{5mm}
\begin{minipage}[b]{3in}
\includegraphics[scale=0.3]{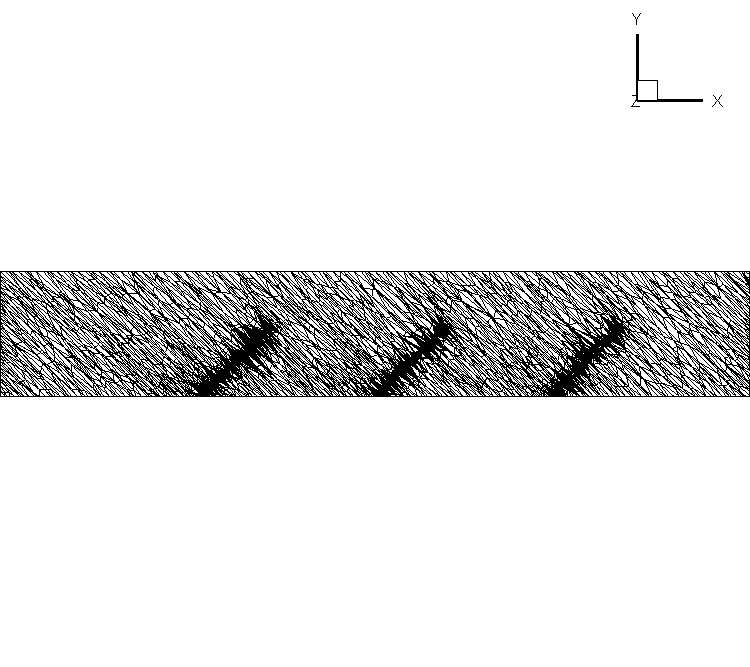}
\centerline{(d): $\M_{DMP+adap}$ mesh, $N=334,547$}
\end{minipage}
}
\caption{Fractured reservoir simulation. Different meshes at cross-section $z=100$ ft.}
\label{ex3-mesh}
\end{figure}

\begin{figure}[thb]
\centering
\hbox{
\hspace{1mm}
\begin{minipage}[b]{3in}
\includegraphics[scale=0.25]{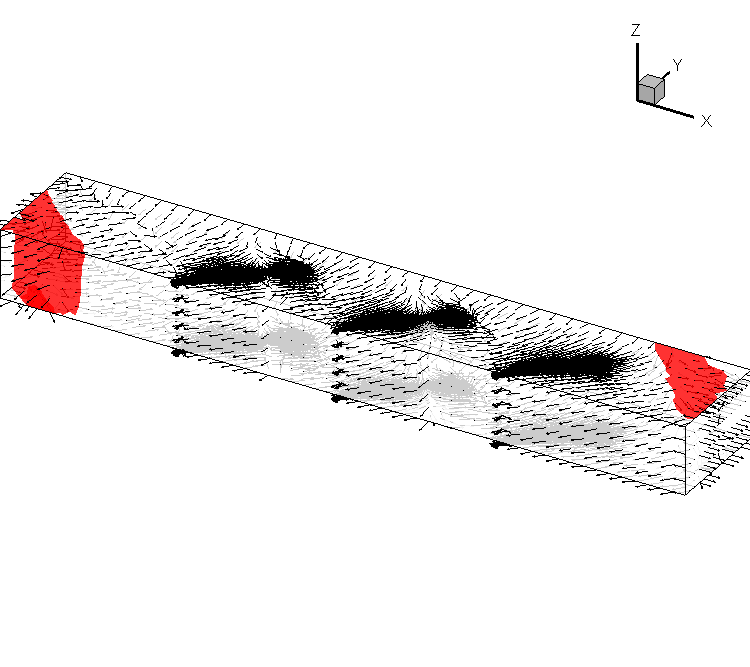}
\centerline{(a): $\M_{id}$ mesh, $P_{max}=3821.57$}
\end{minipage}
\hspace{1mm}
\begin{minipage}[b]{3in}
\includegraphics[scale=0.25]{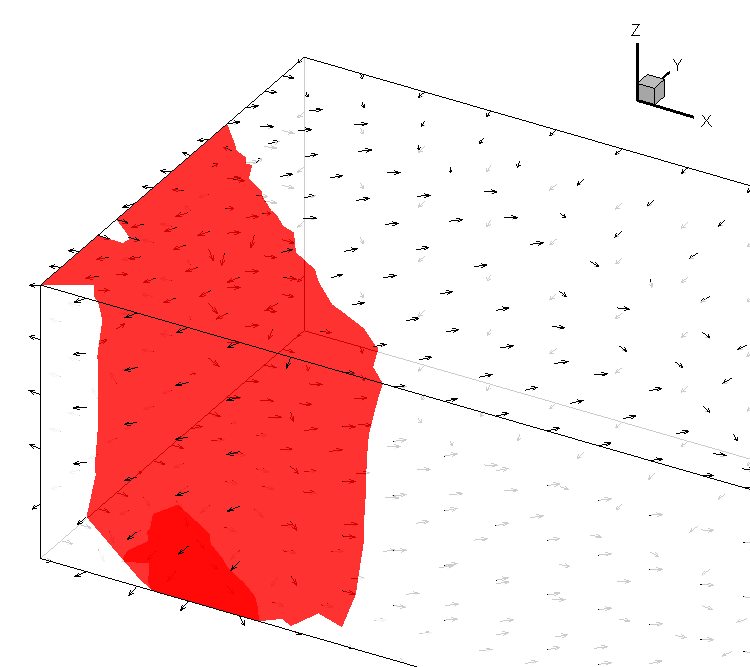}
\centerline{(b): Closer view of (a) at $x=0$ ft}
\end{minipage}
}
\vspace{5mm}
\hbox{
\hspace{1mm}
\begin{minipage}[b]{3in}
\includegraphics[scale=0.25]{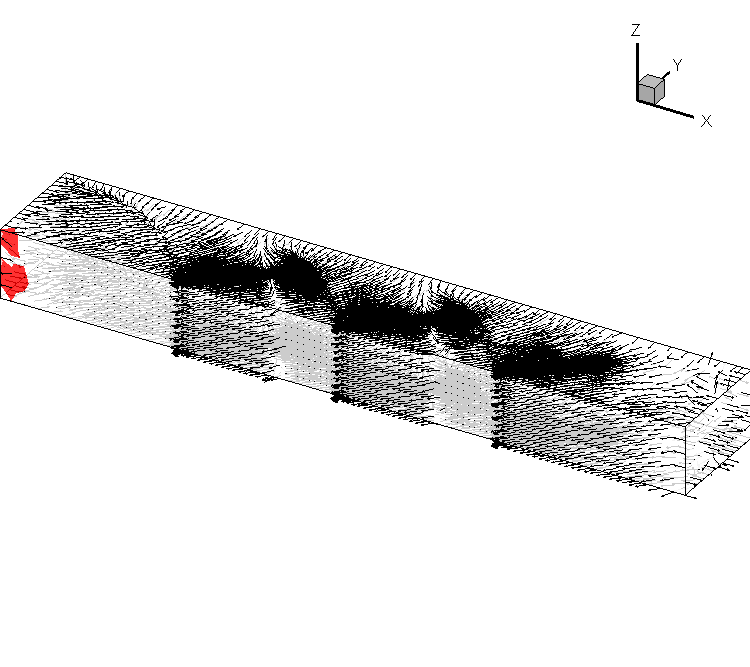}
\centerline{(c): $\M_{adap}$ mesh, $P_{max}=3800.07$}
\end{minipage}
\hspace{1mm}
\begin{minipage}[b]{3in}
\includegraphics[scale=0.25]{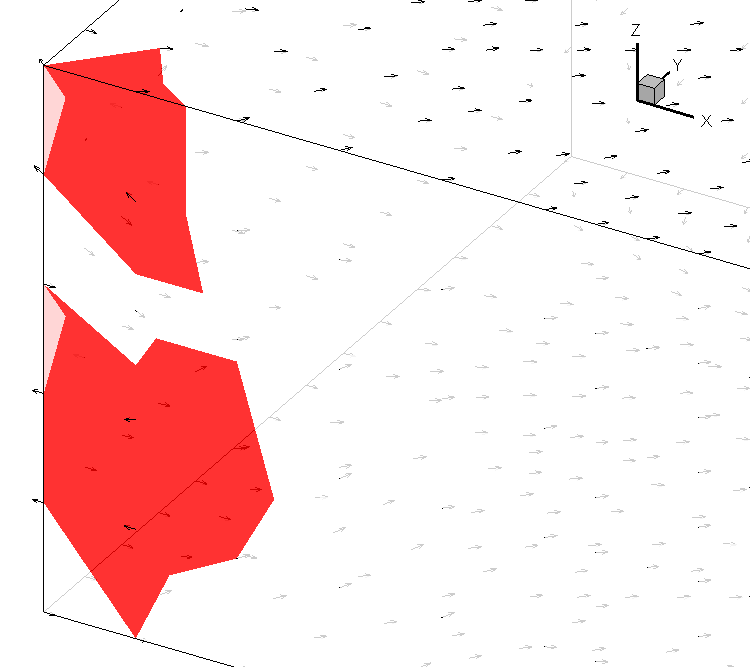}
\centerline{(d): Closer view of (c) at $x=0$ ft}
\end{minipage}
}
\caption{Fractured reservoir simulation. Pressure gradients obtained using $\M_{id}$ and $M_{adap}$ meshes. The red shaded region have unphysical pressure values that are larger than the reservoir pressure.}
\label{ex3-soln-2}
\end{figure}

\begin{figure}[thb]
\centering
\hbox{
\hspace{1mm}
\begin{minipage}[b]{3in}
\includegraphics[scale=0.25]{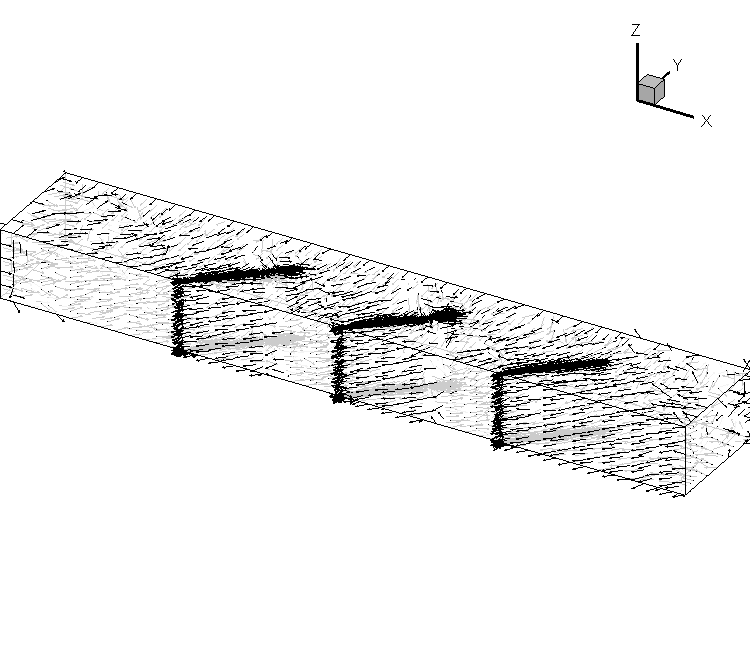}
\centerline{(a): $\M_{DMP+adap}$ mesh, $P_{max}=3800$}
\end{minipage}
\hspace{1mm}
\begin{minipage}[b]{3in}
\includegraphics[scale=0.25]{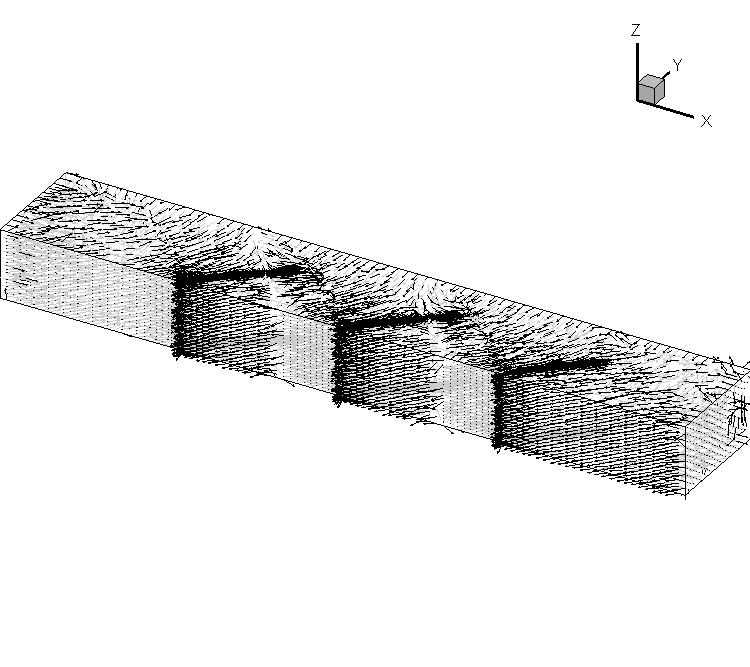}
\centerline{(b): $\M_{DMP+adap}$ mesh, $P_{max}=3800$}
\end{minipage}
}
\caption{Fractured reservoir simulation. Pressure gradients obtained using $\M_{DMP}$ and $\M_{DMP+adap}$ mesh. No unphysical pressure values are observed in the computed solution. }
\label{ex3-soln-2b}
\end{figure}

Different meshes at cross-sections along the fracture at $x=1500$ ft and around the fractures at cross-section of $z=100$ ft are shown in Figs.~\ref{ex3-mesh-f} and \ref{ex3-mesh-z}, respectively. $\M_{DMP}$ and $\M_{DMP+adap}$ meshes clearly are aligned along
the principal diffusion direction both in fracture and in the matrix (where the primary diffusion direction is perpendicular to fractures). 
The alignment of mesh elements helps balance the anisotropic diffusion,
and the numerical solutions obtained using both the $\M_{DMP}$ and $\M_{DMP+adap}$ meshes 
satisfy MP. $\M_{DMP+adap}$ mesh also shows a degree of element concentration around the fractures.
Table~\ref{ex3-result} lists the maximum pressure values for different meshes, which shows that improper meshes can lead to unphysical solutions in the fractured reservoir simulation.

\begin{figure}[thb]
\centering
\hbox{
\hspace{1mm}
\begin{minipage}[b]{3in}
\includegraphics[scale=0.25]{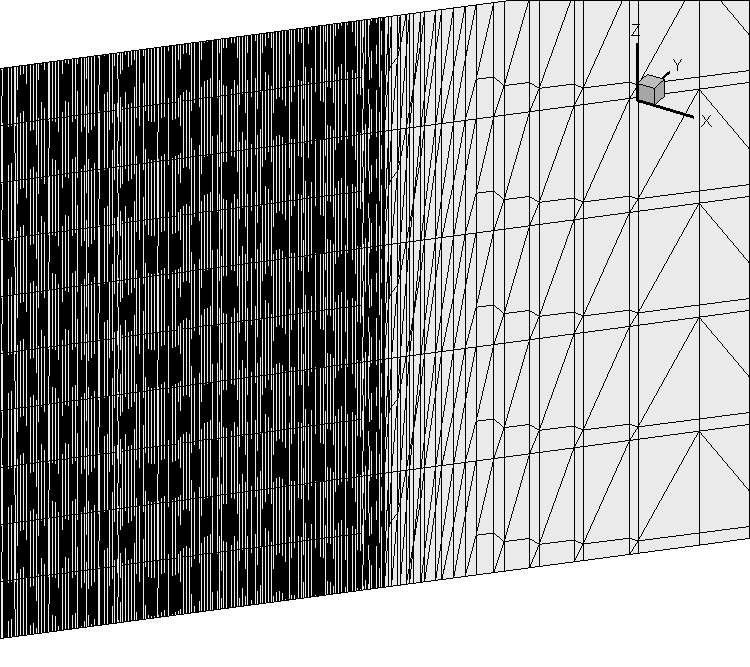}
\centerline{(a): $\M_{id}$ mesh}
\end{minipage}
\hspace{1mm}
\begin{minipage}[b]{3in}
\includegraphics[scale=0.25]{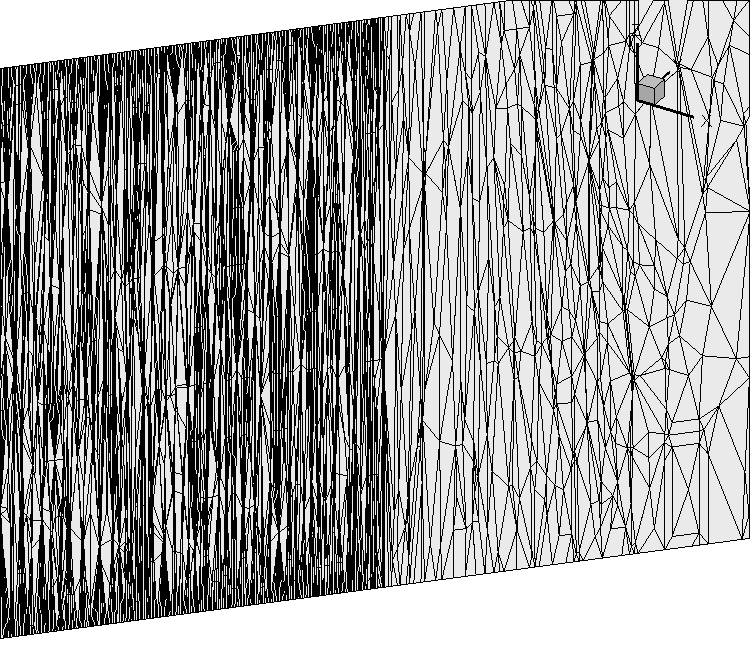}
\centerline{(b): $\M_{adap}$ mesh}
\end{minipage}
}
\vspace{5mm}
\hbox{
\hspace{1mm}
\begin{minipage}[b]{3in}
\includegraphics[scale=0.25]{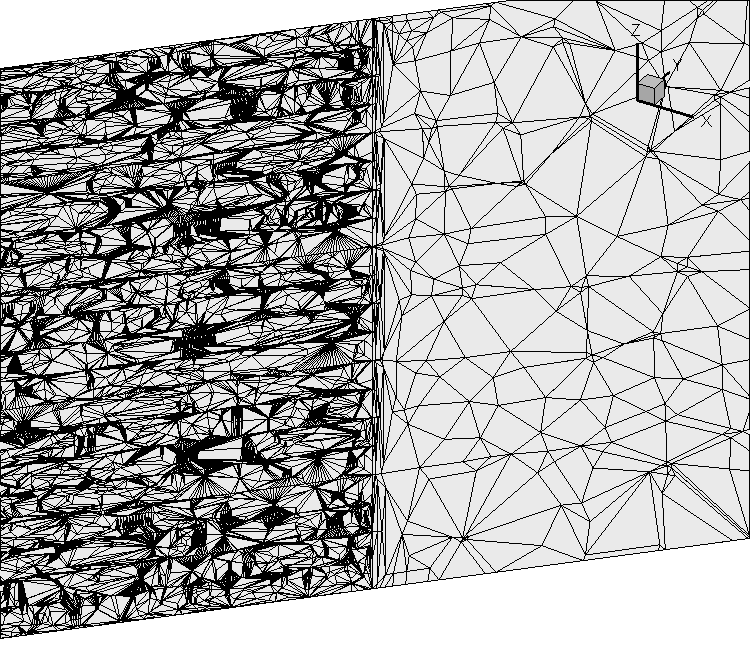}
\centerline{(c): $\M_{DMP}$ mesh}
\end{minipage}
\hspace{1mm}
\begin{minipage}[b]{3in}
\includegraphics[scale=0.25]{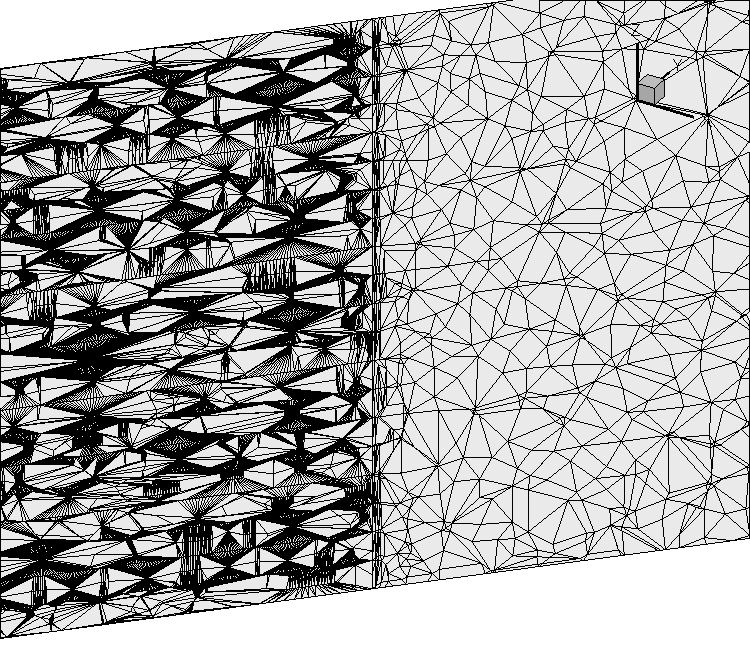}
\centerline{(d): $\M_{DMP+adap}$ mesh}
\end{minipage}
}
\caption{Fractured reservoir simulation. Different meshes at cross-section along the fracture at $x=1500$ ft.}
\label{ex3-mesh-f}
\end{figure}

\begin{figure}[thb]
\centering
\hbox{
\hspace{1mm}
\begin{minipage}[b]{3in}
\includegraphics[scale=0.25]{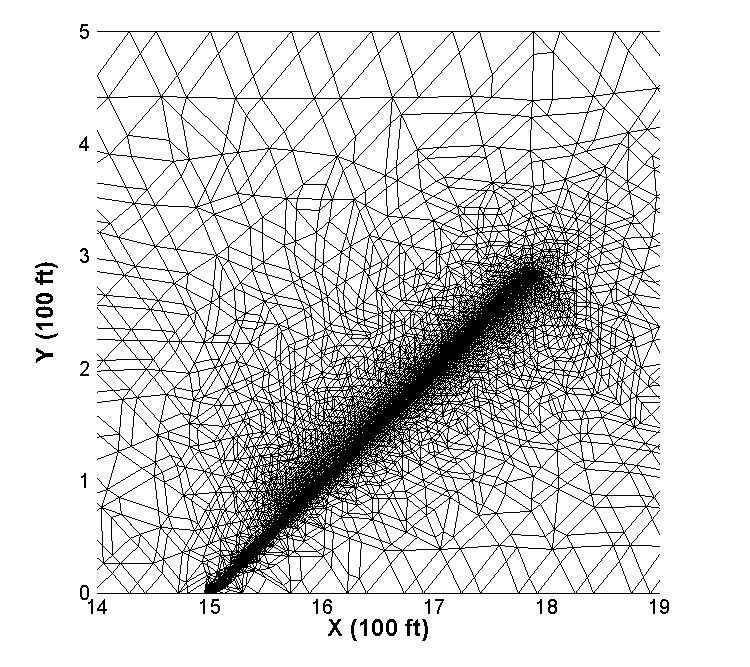}
\centerline{(a): $\M_{id}$ mesh}
\end{minipage}
\hspace{1mm}
\begin{minipage}[b]{3in}
\includegraphics[scale=0.25]{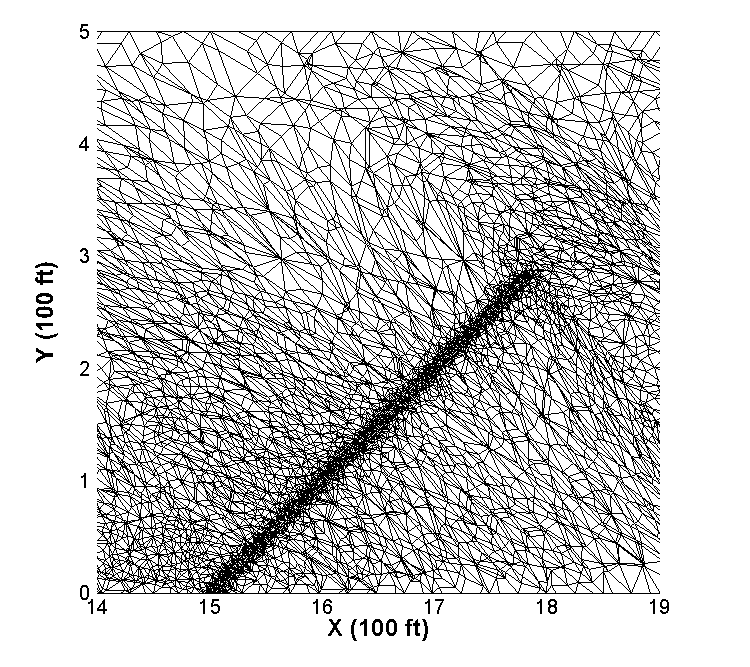}
\centerline{(b): $\M_{adap}$ mesh}
\end{minipage}
}
\vspace{5mm}
\hbox{
\hspace{1mm}
\begin{minipage}[b]{3in}
\includegraphics[scale=0.25]{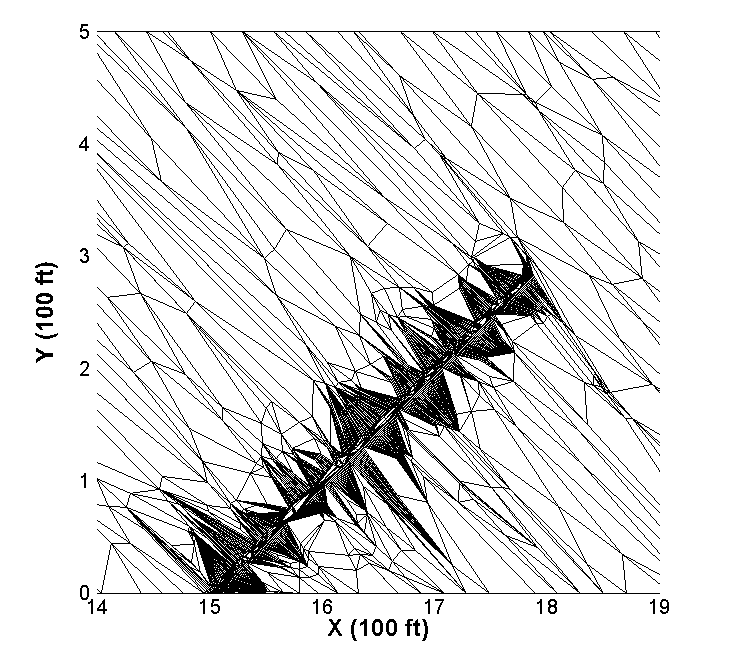}
\centerline{(c): $\M_{DMP}$ mesh}
\end{minipage}
\hspace{1mm}
\begin{minipage}[b]{3in}
\includegraphics[scale=0.25]{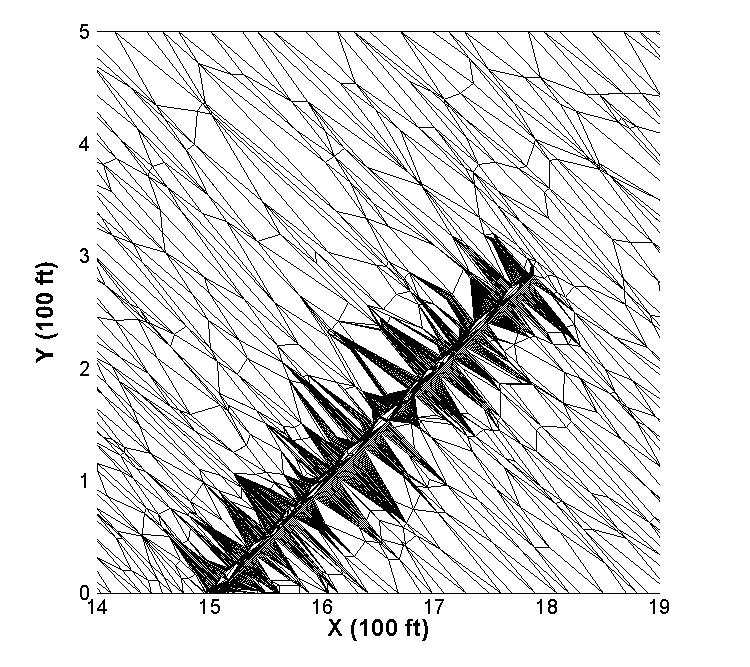}
\centerline{(d): $\M_{DMP+adap}$ mesh}
\end{minipage}
}
\caption{Fractured reservoir simulation. Different meshes at cross-section $z=150$ ft, $1400$ ft $\le x \le 1900$ ft.}
\label{ex3-mesh-z}
\end{figure}

\begin{table}[hbt]
\caption{Fractured reservoir simulation. Maximum pressure values obtained with different meshes.}
\vspace{2pt}
\centering
\begin{tabular}{c|r|l}
\hline \hline
 Mesh & $N$ & $P_{max}$ \\
\hline
$\M_{id}$ & 347,760 & 3821.57  \\
$\M_{adap}$ & 331,134 & 3800.07  \\
$\M_{DMP}$ & 351,206 & 3800  \\
$\M_{DMP+adap}$ & 334,547 & 3800  \\
\hline \hline
\end{tabular}
\label{ex3-result}
\end{table}

\section{Conclusions and comments}
\label{sec-con}

In the previous sections we have studied anisotropic mesh adaptation and maximum principle preservation
for the finite element solution of three-dimensional anisotropic diffusion problems. We have defined a quality measure $Q_{ali}$ in \eqref{meas-ali} that provides a good guidance for MP satisfaction. And we have also developed some new sufficient conditions (cf. (\ref{meshcnd-3}), (\ref{meshcnd-5}), and (\ref{meas-ali3})) that guarantee the linear finite element solution to satisfy MP. Comparison has been performed for four different metric tensors that are used to generate adaptive meshes with the existing software MMG3D.

One of the metric tensors is the identity matrix which typically leads to uniform or almost uniform meshes.
The other three are related to the diffusion matrix and/or
the finite element solution. $\M_{adap}$ (\ref{M-adap}) is based on minimizing $H^1$ semi-norm of linear interpolation
error.  $\M_{adap}$ meshes concentrate elements near the regions where the Hessian of the solution is large,
and gives the smallest error. $\M_{DMP}$ (\ref{M-DMP}) is taken as the inverse of the diffusion matrix and
leads to meshes with elements aligned along the primary diffusion directions, which improves
the MP satisfaction of the finite element solution but gives the largest error among all of the considered metric tensors.
The last metric tensor, $\M_{DMP+adap}$ (\ref{M-DMP+adap}) is proportional to the inverse of the diffusion matrix,
with the coefficient of proportionality being taken as a function to minimize the $H^1$ semi-norm of linear
interpolation error. It provides a good balance between MP preservation ($\M_{DMP}$)
and mesh adaptivity ($\M_{adap}$). Meshes associated with this metric tensor improves the MP satisfaction
of the finite element solution while keeping the error minimal.

The anisotropic mesh adaptation procedure is applied to fractured reservoir simulation in petroleum engineering.
Numerical results show that $\M_{id}$ mesh (with manual adaptation around fractures) can lead to unphysical solutions, $\M_{adap}$ is capable of concentrating mesh elements around the fractures to reduce solution errors but may still violate MP, 
while both $\M_{DMP}$ and $\M_{DMP+adap}$ tend to align elements along the primary diffusion direction
where the permeability is significantly larger than those in other directions and produce no artificial oscillations
in the computed solution.

Overall, the numerical observations we made here for three dimensional problems are consistent with
those for two dimensions reported in \cite{LH10}. However, numerical experience suggests that
it is much harder in 3D to make the mesh to satisfy or closely satisfy the alignment condition (\ref{ali-1}).
This is especially true for the mesh adaptation situation (with the metric tensor depending on the solution)
for which $Q_{ali}$ is relatively large for some elements although $\|Q_{ali}\|_{L^2}$ (which is an average
of $Q_{ali}$) stays relatively small. How to generate better nearly $\M$-uniform meshes for a given metric tensor
in 3D has attracted interest from researchers (e.g. see \cite{Los14, MA14}) and certainly deserves more investigations.


\begin{thebibliography}{10}

\bibitem{ABBM98a}
I.~Aavatsmark, T.~Barkve, {\O}.~B{\o}e, and T.~Mannseth.
\newblock Discretization on unstructured grids for inhomogeneous, anisotropic
  media. {I}. {D}erivation of the methods.
\newblock \textit{ SIAM J. Sci. Comput.}, 19:1700--1716 (electronic), 1998.

\bibitem{ABBM98b}
I.~Aavatsmark, T.~Barkve, {\O}.~B{\o}e, and T.~Mannseth.
\newblock Discretization on unstructured grids for inhomogeneous, anisotropic
  media. {II}. {D}iscussion and numerical results.
\newblock \textit{ SIAM J. Sci. Comput.}, 19:1717--1736 (electronic), 1998.

\bibitem{ABHFDV02}
D.~Ait-Ali-Yahia, G.~Baruzzi, W.G. Habashi, M.~Fortin, J.~Dompierre, and M.G.
  Vallet.
\newblock Anisotropic mesh adaptation: towards user-independent,
  mesh-independent and solver-independent cfd. part ii: structured grids.
\newblock \textit{ Int. J. Numer. Meth. Fluids}, 39:657 -- 673, 2002.

\bibitem{BMT03}
I.I. Bogdanov, V.V. Mourzenko, and J.-F. Thovert.
\newblock Effective permeability of fractured porous media in steady state
  flow.
\newblock \textit{ Water Resources Research}, 39, No. 1, 2003.

\bibitem{BGHLS97}
H.~Borouchaki, P.L. George, P.~Hecht, P.~Laug, and E.~Saletl.
\newblock Delaunay mesh generation governed by metric specification: {P}art
  {I}. {A}lgorithms.
\newblock \textit{ Finite Elements in Analysis and Design}, 25:61--83, 1997.

\bibitem{BGM97}
H.~Borouchaki, P.L. George, and B.~Mohammadi.
\newblock Delaunay mesh generation governed by metric specification: Part {II}.
  applications.
\newblock \textit{ Finite Element in Analysis and Design}, 25:85 -- 109, 1997.

\bibitem{BH96}
F.J. Bossen and P.S. Heckbert.
\newblock A pliant method for anisotropic mesh generation.
\newblock In \textit{ Proceedings, 5th International Meshing roundtable}, pages 63
  -- 74, Albuquerque, NM, 1996. Sandia National Laboratories.
\newblock Sandia Report 96-2301.

\bibitem{BKK08}
J.~Brandts, S.~Korotov, and M.~K\v{r}\'i\v{z}ek.
\newblock The discrete maximum principle for linear simplicial finite element
  approximations of a reaction-diffusion problem.
\newblock \textit{ Linear Algebra and its Applications}, 429:2344--2357, 2008.

\bibitem{CHMP97}
M.J. Castro-D\'{i}az, F.~Hecht, B.~Mohammadi, and O.~Pironneau.
\newblock Anisotropic unstructured mesh adaption for flow simulations.
\newblock \textit{ Int. J. Numer. Meth. Fluids}, 25:475--491, 1997.

\bibitem{CS00}
T.F. Chan and J.~Shen.
\newblock Non-texture inpainting by curvature driven diffusions (cdd).
\newblock \textit{ J. Vis. Commun. Image Rep}, 12:436--449, 2000.

\bibitem{CSV03}
T.F. Chan, J.~Shen, and L.~Vese.
\newblock Variational pde models in image processing.
\newblock \textit{ Not. AMS J.}, 50:14--26, 2003.

\bibitem{CR73}
P.G. Ciarlet and P.-A. Raviart.
\newblock Maximum principle and uniform convergence for the finite element
  method.
\newblock \textit{ Comput. Methods Appl. Mech. Engrg.}, 2:17--31, 1973.

\bibitem{CLER10}
C.L. Cipolla, E.P. Lolon, J.C. Erdle, and B.~Rubin.
\newblock Reservoir modeling in shale-gas reservoirs.
\newblock Paper SPE 125530, presented at the SPE Eastern Regional Meeting held
  in Charleston, West Virginia, USA, 23-25, September 2009, paper peer approved
  1 March 2010.

\bibitem{CLET09}
C.L. Cipolla, E.P. Lolon, J.C. Erdle, and V.~Tathed.
\newblock Modeling well performance in shale-gas reservoirs.
\newblock Paper SPE 125532, presented at the SPE/EAGE Reservoir
  Characterization and Simulation Conference held in Abu Dhabi, UAE, 19-21,
  October 2009.

\bibitem{CSW95}
P.I. Crumpton, G.J. Shaw, and A.F. Ware.
\newblock Discretisation and multigrid solution of elliptic equations with
  mixed derivative terms and strongly discontinuous coefficients.
\newblock \textit{ J. Comput. Phys.}, 116:343--358, 1995.

\bibitem{mmg3d}
C.~Dobrzynski and P.~Frey.
\newblock 2012.
\newblock Anisotropic Tetrahedral Remesher/Moving Mesh Generation software,
  available at: http://www.math.u-bordeaux1.fr/~dobrzyns/logiciels/mmg3d.php.

\bibitem{DVBFH02}
J.~Dompierre, M.G. Vallet, Y.~Bourgault, M.~Fortin, and W.G. Habashi.
\newblock Anisotropic mesh adaptation: towards user-independent,
  mesh-independent and solver-independent {CFD}. {P}art {III}: {U}nstructured
  meshes.
\newblock \textit{ Int. J. Numer. Meth. Fluids}, 39:675 -- 702, 2002.

\bibitem{DDS04}
A.~Dr\v{a}g\v{a}nescu, T.F. Dupont, and L.R. Scott.
\newblock Failure of the discrete maximum principle for an elliptic finite
  element problem.
\newblock \textit{ Math. Comp.}, 74(249):1--23, 2004.

\bibitem{FMIB13}
C.M. Freeman, G.~Moridis, D.~Ilk, and T.A. Blasingame.
\newblock A numerical study of performance for tight gas and shale gas
  reservoir systems.
\newblock \textit{ Journal of Petroleum Science and Engineering}, 108:22--39, 2013.

\bibitem{GS98}
R.V. Garimella and M.S. Shephard.
\newblock Boundary layer meshing for viscous flows in complex domain.
\newblock In \textit{ Proceedings, 7th International Meshing Roundtable}, pages
  107--118, Albuquerque, NM, 1998. Sandia National Laboratories.

\bibitem{GL09}
S.~G{\"u}nter and K.~Lackner.
\newblock A mixed implicit-explicit finite difference scheme for heat transport
  in magnetised plasmas.
\newblock \textit{ J. Comput. Phys.}, 228:282--293, 2009.

\bibitem{GLT07}
S.~G{\"u}nter, K.~Lackner, and C.~Tichmann.
\newblock Finite element and higher order difference formulations for modelling
  heat transport in magnetised plasmas.
\newblock \textit{ J. Comput. Phys.}, 226:2306--2316, 2007.

\bibitem{GYKL05}
S.~G{\"u}nter, Q.~Yu, J.~Kruger, and K.~Lackner.
\newblock Modelling of heat transport in magnetised plasmas using non-aligned
  coordinates.
\newblock \textit{ J. Comput. Phys.}, 209:354--370, 2005.

\bibitem{HDBAFV00}
W.G. Habashi, J.~Dompierre, Y.~Bourgault, D.~Ait-Ali-Yahia, M.~Fortin, and M.G.
  Vallet.
\newblock Anisotropic mesh adaptation: towards user-independent,
  mesh-independent and solver-independent {CFD}. {P}art {I}: {G}eneral
  principles.
\newblock \textit{ Int. J. Numer. Meth. Fluids}, 32:725--744, 2000.

\bibitem{bamg}
F.~Hecht.
\newblock Bidimensional anisotropic mesh generator software (bamg).
\newblock \textit{ http://www.ann.jussieu.fr/hecht/ftp/bamg/bamg-v1.01.tar.gz},
  2010.

\bibitem{FFem}
F.~Hecht.
\newblock New development in freefem++.
\newblock \textit{ J. Numer. Math.}, 20, no. 3-4:251--265, 2012.
\newblock 65Y15.

\bibitem{Hua01b}
W.~Huang.
\newblock Variational mesh adaptation: isotropy and equidistribution.
\newblock \textit{ J. Comput. Phys.}, 174:903--924, 2001.

\bibitem{Hua05c}
W.~Huang.
\newblock Measuring mesh qualities and application to variational mesh adaptation.
\newblock \textit{ SIAM. J. Sci. Comput.}, 26:1643--1666, 2005.

\bibitem{Hua05b}
W.~Huang.
\newblock Metric tensors for anisotropic mesh generation.
\newblock \textit{ J. Comput. Phys.}, 204:633--665, 2005.

\bibitem{Hua06}
W.~Huang.
\newblock Mathematical principles of anisotropic mesh adaptation.
\newblock \textit{ Comm. Comput. Phys.}, 1:276--310, 2006.

\bibitem{Hua11}
W.~Huang.
\newblock Discrete maximum principle and a delaunay-type mesh condition for
  linear finite element approximations of two-dimensional anisotropic diffusion
  problems.
\newblock \textit{ Numerical Mathematics: Theory, Methods and Applications},
  4:319--334, 2011.

\bibitem{HKL10-2}
W.~Huang, L.~Kamenski, and X.~Li.
\newblock Anisotropic mesh adaptation for variational problems using error
  estimation based on hierarchical bases.
\newblock \textit{ Canadian Applied Mathematics Quarterly (Special issue for the
  30th anniversary of CAIMS)}, 17:501--522, 2009.

\bibitem{KKK07}
J.~Kar\'atson, S.~Korotov, and M.~K\v{r}\'i\v{z}ek.
\newblock On discrete maximum principles for nonlinear elliptic problems.
\newblock \textit{ Mathematics and Computers in Simulation}, 76:99--108, 2007.

\bibitem{KM09}
D.A. Karras and G.B. Mertzios.
\newblock New pde-based methods for image enhancement using som and bayesian
  inference in various discretization schemes.
\newblock \textit{ Meas. Sci. Technol.}, 20:104012, 2009.


\bibitem{Knu03}
P. M. Knupp.
\newblock { Algebraic mesh quality metrics for unstructured initial meshes}.
\newblock \textit{ Fin. Elem. Anal. Des.}, 39: 217--241, 2003.

\bibitem{Knu12}
P. M. Knupp.
\newblock { Introducing the target-matrix paradigm for mesh optimization via node-movement}.
\newblock \textit{ Eng. Comput.}, 28: 419--429, 2012.

\bibitem{KL95}
M.~K\v{r}\'i\v{z}ek and Q.~Lin.
\newblock On diagonal dominance of stiffness matrices in 3d.
\newblock \textit{ East-West J. Numer. Math.}, 3:59--69, 1995.

\bibitem{KSS09}
D.~Kuzmin, M.J. Shashkov, and D.~Svyatskiy.
\newblock A constrained finite element method satisfying the discrete maximum
  principle for anisotropic diffusion problems.
\newblock \textit{ J. Comput. Phys.}, 228:3448--3463, 2009.

\bibitem{Let92}
F.W. Letniowski.
\newblock Three-dimensional delaunay triangulations for finite element
  approximations to a second-order diffusion operator.
\newblock \textit{ SIAM J. Sci. Stat. Comput.}, 13:765--770, 1992.

\bibitem{LH10}
X.~Li and W.~Huang.
\newblock An anisotropic mesh adaptation method for the finite element solution
  of heterogeneous anisotropic diffusion problems.
\newblock \textit{ J. Comput. Phys.}, 229:8072--8094, 2010.

\bibitem{LH13}
X.~Li and W.~Huang.
\newblock Maximum principle for the finite element solution of time-dependent
  anisotropic diffusion problems.
\newblock \textit{ Numer. Meth. PDEs}, 29:1963--1985, 2013.

\bibitem{LSS07}
X.~Li, D.~Svyatskiy, and M.~Shashkov.
\newblock Mesh adaptation and discrete maximum principle for {2D} anisotropic
  diffusion problems.
\newblock \textit{ LANL technical report}, LA-UR 10-01227, 2007.

\bibitem{LSSV07}
K.~Lipnikov, M.~Shashkov, D.~Svyatskiy, and Y.~Vassilevski.
\newblock Monotone finite volume schemes for diffusion equations on
  unstructured triangular and shape-regular polygonal meshes.
\newblock \textit{ J. Comput. Phys.}, 227:492--512, 2007.

\bibitem{LS08}
R.~Liska and M.~Shashkov.
\newblock Enforcing the discrete maximum principle for linear finite element
  solutions of second-order elliptic problems.
\newblock \textit{ Communications in Computational Physics}, 3:852--877, 2008.

\bibitem{Los14}
A.~Loseille.
\newblock Metric-orthogonal anisotropic mesh generation.
\newblock \textit{ Procedia Engineering}, 82:403--415, 2014.

\bibitem{LHQ14}
C.~Lu, W.~Huang, and J.~Qiu.
\newblock Maximum principle in linear finite element approximations of
  anisotropic diffusion-convection-reaction problems.
\newblock \textit{ Numer. Math.}, 127:515--537, 2014.

\bibitem{LL14}
I.~Lunati and S.H. Lee.
\newblock A dual-tube model for gas dynamics in fractured nanoporous shale
  formations.
\newblock \textit{ J. Fluid Mech.}, 757:943--971, 2014.

\bibitem{MA14}
D.~Marcum and F.~Alauzet.
\newblock Aligned metric-based anisotropic solution adaptive mesh generation.
\newblock \textit{ Procedia Engineering}, 82:428--444, 2014.

\bibitem{MJL14}
L.~Mi, H.~Jiang, and J.~Li.
\newblock The impact of diffusion type on multiscale discrete fracture model
  numerical simulation for shale gas.
\newblock \textit{ Journal of Natural Gas Science and Engineering}, 20:74--81,
  2014.

\bibitem{MD06}
M.J. Mlacnik and L.J. Durlofsky.
\newblock Unstructured grid optimization for improved monotonicity of discrete
  solutions of elliptic equations with highly anisotropic coefficients.
\newblock \textit{ J. Comput. Phys.}, 216:337--361, 2006.

\bibitem{NW00}
K.~Nishikawa and M.~Wakatani.
\newblock \textit{ Plasma Physics}.
\newblock Springer-Verlag Berlin Heidelberg, New York, 2000.

\bibitem{OFMB12}
O.M. Olorode, C.M. Freeman, G.J. Moridis, and T.A. Blasingame.
\newblock High-resolution numerical modeling of complex and irregular fracture
  patterns in shale gas and tight gas reservoirs.
\newblock Paper SPE 152482, presented at the SPE Latin American and Caribbean
  Petroleum Engineering Conference held in Mexico City, Mexico, 16-18, April
  2012.

\bibitem{PVMZ87}
J.~Peraire, M.~Vahdati, K.~Morgan, and O.C. Zienkiewicz.
\newblock Adaptive remeshing for compressible flow computations.
\newblock \textit{ J. Comput. Phys.}, 72:449--466, 1987.

\bibitem{PM90}
P.~Perona and J.~Malik.
\newblock Scale-space and edge detection using anisotropic diffusion.
\newblock \textit{ IEEE Transactions on Pattern Analysis and Machine Intelligence},
  12:629--639, 1990.

\bibitem{LePot09}
C.~Le Potier.
\newblock A nonlinear finite volume scheme satisfying maximum and minimum
  principles for diffusion operators.
\newblock \textit{ Int. J. Finite Vol.}, 6:20, 2009.

\bibitem{Rubin10}
B.~Rubin.
\newblock Accurate simulation of non-darcy flow in stimulated fractured shale
  reservoir.
\newblock Paper SPE 132093, presented at the SPE Western Regional Meeting held
  in Anaheim, California, USA, 27-29, May 2010.

\bibitem{SH07}
P.~Sharma and G.W. Hammett.
\newblock Preserving monotonicity in anisotropic diffusion.
\newblock \textit{ Journal of Computational Physics}, 227:123--142, 2007.

\bibitem{SK12}
D.~Silin and T.~Kneafsey.
\newblock Shale gas: nanometer-scale observations and well modelling.
\newblock Paper SPE 149489, presented at the Canadian Unconventional Resources
  Conference held in Calgary, CA, 15-17 November 2011, paper peer approved 7
  June 2012.

\bibitem{VPS13}
M.~Vu, A.~Pouya, and D.M. Seyedi.
\newblock Modelling of steady-state fluid flow in 3d fractured isotropic porous
  media: application to effective permeability calculation.
\newblock \textit{ Int. J. Numer. Anal. Meth. Geomech}, 37:2257--2277, 2013.

\bibitem{WCJR14}
Q.~Wang, X.~Chen, A.N. Jha, and H.~Rogers.
\newblock Natural gas from shale formation – the evolution, evidences and
  challenges of shale gas revolution in united states.
\newblock \textit{ Renewable and Sustainable Energy Reviews}, 30:1--28, 2014.

\bibitem{WaZh11}
J.~Wang and R.~Zhang.
\newblock {  Maximum principle for P1-conforming finite element approximations of quasi-linear
second order elliptic equations}.
\newblock \textit{ SIAM J. Numer. Anal.}, 50: 626--642, 2012.

\bibitem{Wei98}
J.~Weickert.
\newblock \textit{ Anisotropic diffusion in image processing}.
\newblock Teubner-Verlag, Stuttgart, Germany, 1998.

\bibitem{XZ99}
J.~Xu and L.~ Zikatanov.
\newblock { A monotone finite element scheme for convection-diffusion equations}.
\newblock \textit{ Math. Comp.}, 69: 1429--1446, 1999.

\bibitem{YS00}
S.~Yamakawa and K.~Shimada.
\newblock High quality anisotropic tetrahedral mesh generation via ellipsoidal
  bubble packing.
\newblock In \textit{ Proceedings, 9th International Meshing Roundtable}, pages
  263--273, Albuquerque, NM, 2000. Sandia National Laboratories.
\newblock Sandia Report 2000-2207.

\bibitem{YuSh2008}
G.~Yuan and Z.~Sheng.
\newblock { Monotone finite volume schemes for diffusion equations on polygonal meshes}.
\newblock \textit{ J. Comput. Phys.}, 227: 6288--6312, 2008.

\bibitem{ZZS2013}
Y.~Zhang, X.~Zhang, and C.-W.~Shu.
\newblock { Maximum-principle-satisfying second order discontinuous Galerkin schemes for convection-diffusion
equations on triangular meshes}.
\newblock \textit{ J. Comput. Phys.}, 234:295--316, 2013.

\end{thebibliography}
\end{document}